\documentclass[12pt]{amsart}
\usepackage{euscript,epsf,amssymb,xr}

\let\oldmarginpar\marginpar
\renewcommand\marginpar[1]
{\oldmarginpar{\tiny\bf \begin{flushleft} #1 \end{flushleft}}}

\input xy
\xyoption{all}

\newcommand{\la}{\langle}
\newcommand{\ra}{\rangle}

\newtheorem{theorem}{\bf Theorem}[section]
\newtheorem{lemma}[theorem]{\bf Lemma}

\newtheorem{corollary}[theorem]{\bf Corollary}

\newcommand{\CC}{{\Bbb C}}

\newcommand{\CP}{{\Bbb CP}}

\newcommand{\FF}{{\Bbb F}}

\newcommand{\NN}{{\Bbb N}}

\newcommand{\RR}{{\Bbb R}}

\newcommand{\ZZ}{{\Bbb Z}}

\newcommand{\CCC}{{\mathcal C}}

\newcommand{\ggreat}{>\kern-.7ex>}
\newcommand{\ssmall}{<\kern-.7ex<}
\newcommand{\qu}{/\kern-.7ex/}
\newcommand{\exh}{\to\kern-1.8ex\to}

\newcommand{\cC}{{\EuScript{C}}}
\newcommand{\dD}{{\EuScript{D}}}

\newcommand{\fF}{{\EuScript{F}}}
\newcommand{\gG}{{\EuScript{G}}}
\newcommand{\hH}{{\EuScript{H}}}

\newcommand{\sS}{{\EuScript{S}}}

\newcommand{\GL}{\operatorname{GL}}

\newcommand{\Aut}{\operatorname{Aut}}

\newcommand{\depth}{\operatorname{depth}}

\newcommand{\Diff}{\operatorname{Diff}}
\newcommand{\GCD}{\operatorname{GCD}}
\newcommand{\stable}{\operatorname{st}}

\newcommand{\Jord}{\operatorname{Jord}}
\newcommand{\Ker}{\operatorname{Ker}}

\renewcommand{\O}{\operatorname{O}}
\newcommand{\sd}{\operatorname{sd}}

\newcommand{\Tr}{\operatorname{Tr}}

\newcommand{\ov}{\overline}
\newcommand{\un}{\underline}

\newcommand{\vt}{\vartriangleleft}

\title{Finite group actions on manifolds without odd cohomology}

\author{Ignasi Mundet i Riera}
\address{Departament d'\`Algebra i Geometria\\
Facultat de Matem\`atiques\\
Universitat de Barcelona\\
Gran Via de les Corts Catalanes 585\\
08007 Barcelona \\
Spain}
\email{ignasi.mundet@ub.edu}

\date{May 23, 2014}


\begin{document}

\maketitle

\begin{abstract}
Let $X$ be a compact smooth manifold,
possibly with boundary. Denote by $X_1,\dots,X_r$ the connected components
of $X$. Assume that $H^*(X;\ZZ)$
is torsion free and supported in even degrees. We prove that
there exists a constant $C$ such that any finite group $G$
acting smoothly and effectively on $X$
has an abelian subgroup $A$ of index at most $C$, which can be generated
by at most $\sum_i[\dim X_i/2]$ elements, and which satisfies $\chi(X_i^A)=\chi(X_i)$
for every $i$.
In particular this confirms, for all such manifolds $X$, a conjecture of \'Etienne Ghys.
An essential ingredient of the proof is a result on finite groups by
Alexandre Turull which uses the classification of finite simple groups.
\end{abstract}

\section{Introduction}

\subsection{Main results}
In this paper we study smooth actions of finite groups on smooth compact manifolds
which may have boundary. We will implicitly assume, unless we say the contrary, all manifolds
to be smooth and compact, but not necessarily connected,
all groups to be finite (except for some obvious exceptions) and
all actions of groups on manifolds to be smooth.

We say that a manifold $X$ has no odd cohomology if its integral cohomology is torsion free and
supported in even degrees (note that if $X$ is orientable and closed then the assumption that
$H^*(X;\ZZ)$ is supported in even degrees implies, by Poincar\'e duality and the universal
coefficient theorem, that the cohomology is torsion free).

The following is the main result of this paper.

\begin{theorem}
\label{thm:main}
Let $X$ be a 
manifold without odd cohomology, with connected components $X_1,\dots,X_r$.
There exists an integer $C\geq 1$ such
that any finite  group $G$ acting effectively on $X$
has an abelian subgroup $A$ of index at most $C$, which can be generated
by at most $\sum_i[\dim X_i/2]$ elements, and such that for any
$i$ we have $\chi(X_i^A)=\chi(X_i)$.
In particular, $A$ has fixed points in every connected component of $X$.
\end{theorem}

It will be clear from the proof that the theorem can be extended to manifolds
with vanishing odd dimensional Betti numbers and with torsion in their cohomology,
as long as the order of the torsion
and the order of the finite group acting on the manifold are coprime.
The constant $C$ in the statement of the theorem
can be chosen to depend only on the dimension of $X$ and the
sum of the Betti numbers of $X$,
although we don't give in this paper any formula for it, nor do we try to optimize
our arguments so as to find the best value for $C$.

Theorem \ref{thm:main} applies for example to contractible manifolds and to
even dimensional integral homology spheres. More generally,
an abundant collection of manifolds
to which Theorem \ref{thm:main} applies are closed manifolds admitting
a Morse function all of whose critical points have even index. These include
closed compact symplectic manifolds endowed with a Hamiltonian action of $S^1$ with
isolated fixed points (the moment map is in this case a Morse function with
critical points of even index), such as
complex flag manifolds or toric varieties.
The class of manifolds without odd cohomology is closed with respect to products,
locally trivial fibrations with simply connected base, and, in the case
of closed manifolds, with respect to connected sums.

Theorem \ref{thm:main} proves for manifolds without odd cohomology
a conjecture of Ghys (see Question 13.1 in \cite{F}\footnote{This conjecture
was discussed in several talks by Ghys \cite{G} (I thank \'E. Ghys for this information),
but apparently it was in \cite{F} when it appeared in print for the first time.})
which states that
for any compact manifold $X$ there exists some $C\geq 1$ with the property that any finite
group $G$ acting effectively on $X$ contains an abelian subgroup of index at most $C$
(for the case in which $X$ is a sphere
this has also been conjectured by Zimmermann, see \cite{Z}). By the arguments in
\cite[\S 2.3]{M1}, Theorem  \ref{thm:main} also implies Ghys' conjecture for any
manifold admitting a compact covering without odd cohomology (e.g. even dimensional
real projective spaces).

\subsection{Some consequences of Theorem \ref{thm:main}}
For any action of a group $G$ on a manifold $X$ we denote by $G_x$ the
stabilizer of $x\in X$. Since each connected component of a manifold without odd
cohomology has nonzero Euler characteristic, Theorem \ref{thm:main} implies the following.

\begin{corollary}
\label{cor:existence-almost-fixed-point}
For any manifold $X$ without odd cohomology there exists a constant $C$
such that for any action of a finite group $G$ on $X$, any connected
component of $X$ contains a point $x$ such that $[G:G_x]\leq C$.
\end{corollary}

This corollary is nontrivial even for disks, since
disks of high enough dimension support finite group actions without
fixed points. This is a classic theme
in the theory finite transformation groups. The first example of a finite
group action on a disk with empty fixed point set was found
by Floyd and Richardson (see \cite{FR} and also \cite[Chap. I, \S8]{Br}),
and a complete characterization of which finite groups admit smooth actions
on disks without fixed points was given by Oliver
(see Theorem 7 of \cite{O}).

For actions on complex
projective spaces, one can make the following quantitative statement:
for any $C>1$ there exists some $C'$ such that if a finite group $G$
has no abelian subgroup of index at most $C'$, then $G$ acts (linearly)
on some complex projective space in such a way that the isotropy group of
any point has index at least $C$. This follows from a result of Isaacs and
Passman \cite{IP} and the fact that if a representation $G\to\GL(n,\CC)$
is irreducible then the induced action on $\CP^{n-1}$ has the property that
for any $x\in \CP^{n-1}$ we have $[G:G_x]\geq n$.

In \cite{M2} Corollary \ref{cor:existence-almost-fixed-point} has been
proved for all compact 4-manifolds with nonzero Euler characteristic.
It seems reasonable to expect Corollary \ref{cor:existence-almost-fixed-point}
to be true for all compact manifolds with nonzero Euler characteristic, but the
methods in this paper unfortunately fall short of proving such a statement:
Lemma \ref{lemma:no-odd-cohomology}, which
plays a key role in several steps of our arguments,
uses crucially the hypothesis that the manifold has no odd cohomology.

Another immediate consequence of Theorem \ref{thm:main} is the following fact:
for any manifold $X$ without odd cohomology there exists a constant $C$
such that any finite group acting effectively on $X$ can be generated by
at most $C$ elements. However, this result is not really new, as the following
more general result is known to be true\footnote{I thank A. Jaikin and E. Khukhro for explaining Theorem \ref{thm:cota-nombre-generadors} and its proof to me.}:

\begin{theorem}
\label{thm:cota-nombre-generadors}
For any compact manifold $X$ there exists a constant $C$ such that any finite group
$G$ acting effectively on $X$ can be generated by at most $C$ elements.
\end{theorem}
\begin{proof}
By the main result in \cite{MS}, there exists an integer $r$ such that for any prime $p$
any elementary $p$-group acting effectively on $X$ has rank at most $r$.
Suppose that $\Gamma$ is a $p$-group acting effectively on $X$; let $\Gamma_0$ be a maximal
abelian normal subgroup of $\Gamma$. The action by conjugation identifies $\Gamma/\Gamma_0$
with a subgroup of $\Aut(\Gamma_0)$. Since $\Gamma_0$ can be generated by at most $r$ elements,
the Gorchakov--Hall--Merzlyakov--Roseblade lemma (see e.g. Lemma 5 in \cite{Ro}) implies
that $\Gamma/\Gamma_0$ can be generated by at most $r(5r-1)/2$ elements.
Hence, $\Gamma$ can be generated by at most $r(5r+1)/2$ elements.
According to a theorem proved independently by Guralnick and Lucchini \cite{Gu,L},
if all Sylow subgroups of a finite group $G$ can be generated by at most $k$ elements,
then $G$ itself can be generated by at most $k+1$ elements. This implies that any finite group
acting effectively on $X$ can be generated by at most $r(5r+1)/2+1$ elements.
\end{proof}

Note that both Guralnick and Lucchini base their arguments in \cite{Gu,L}
on the classification of the finite simple groups (CFSG).
The combination of Ghys' conjecture with \cite{MS} provides an alternative, more geometric approach to Theorem \ref{thm:cota-nombre-generadors}.
On the other hand, a proof of Ghys' conjecture avoiding
the CFSG would very likely provide a much more elementary
proof of Theorem \ref{thm:cota-nombre-generadors} than the one based on Guralnick and Lucchini's theorem. This is certainly true for the proofs of partial
cases of the conjecture given in \cite{M1,M2} but not for the approach in the present paper,
which uses the CFSG through the results in \cite{MT}.

\subsection{Some ingredients of the proof}
The main tool in the proof of Theorem \ref{thm:main}
is the notion of $K$-stable action of an abelian group on $X$,
where $K$ is a positive real number. An action of an abelian $p$-group $A$ on
a manifold $X$
is $K$-stable if {\it (i)} the action of $A$ on $H^*(X;\FF_p)$ is unipotent, {\it (ii)}
$\chi(X^{A_0})=\chi(X)$ for any subgroup $A_0$ of $A$, and {\it (iii)}
for any subgroup $A_0\subseteq A$ of index smaller than $K$ each connected component of $X^A$ is
a connected component of $X^{A_0}$. An action of an abelian group $A$ on $X$ is $K$-stable
if for every prime $p$ the action of the $p$-part $A_p\subseteq A$ is $K$-stable.
When the action of $A$ is clear from the context, we simply say that $A$ is $K$-stable
(e.g., if a group $G$ acts on $X$, we say that a subgroup $A\subseteq G$ is $K$-stable
if the restriction of the action to $A$ is $K$-stable).

We prove in this paper several properties enjoyed by this notion
in the context of actions on manifolds $X$ without odd cohomology:
{\it (a)} for any $K$ there
exists some $\Lambda(K)\geq 1$ 
such that any abelian group $A$ acting on $X$
has a $K$-stable subgroup of index at most $\Lambda(K)$ (Theorem \ref{thm:existence-stable-subgroups} --- this is valid on arbitrary compact manifolds);
{\it (b)} there exists
some number $K_{\chi}$ 
such that for any $K\geq K_{\chi}$ and any $K$-stable
action of an abelian group $A$ on $X$ we have $\chi(X^A)=\chi(X)$
(Theorem \ref{thm:good-Gamma-has-gamma}); {\it (c)} if
$X$ is connected then there exists
some constant $e$ 
such that: if $K>e$ and $A_1,\dots,A_s$ are $K$-stable abelian subgroups of
a group $G$ acting on $X$
and satisfying $X^{A_1}\cap\dots\cap X^{A_s}\neq\emptyset$, then there
exists a $(K/e)$-stable abelian subgroup $A\subseteq G$ such that
$X^{A_1}\cap\dots\cap X^{A_s}=X^A$ (this is
Theorem \ref{thm:intersection-stabilizes-any-group}; we emphasize that $s$ is arbitrary).
The function $K\mapsto\Lambda(K)$ and the constants $K_{\chi}$ and $e$ in these statements
depend only on $X$.

These results use in an essential way the
smoothness of the group actions,
through the following consequences:
(1) any smooth action of a finite group $G$ on $X$
is locally linear (this is Lemma \ref{lemma:linearization}), so in particular in a neighborhood of any point of $X$ fixed
by a subgroup $\Gamma\subseteq G$ the action of $\Gamma$ is equivalent to
a linear representation of $\Gamma$, and (2)
the isomorphism class of this representation is locally constant on $X^{\Gamma}$
(this is used in Lemma \ref{lemma:existeix-element-generic-p}).
Besides, the smoothness of the action is used in this paper
through the existence of equivariant triangulations,
e.g. in Lemma \ref{lemma:one-big-stabiliser}.

Properties {\it (a)} and {\it (b)} imply, with some additional work, the following.

\begin{theorem}
\label{thm:main-abelian}
Let $X$ be a 
manifold without odd
cohomology and with connected components $X_1,\dots,X_r$.
There exists an integer $C_a\geq 1$ such that any finite abelian group
$A'$ acting effectively on $X$ contains a subgroup $A$ of index at most $C_a$ satisfying
$\chi(X_i^{A})=\chi(X_i)$ for every $i$.
Furthermore $A$ can be generated by at most $\sum_i[\dim X_i/2]$ elements.
\end{theorem}

For cyclic groups, Theorem \ref{thm:main-abelian} follows from
Lemma \ref{lemma:Minkowski} (which reduces the question to actions which are trivial on
cohomology) and Lemma \ref{lemma:Euler-fixed-point-set-cyclic} (which is a
standard analogue of Lefschetz's fixed point theorem).
For noncyclic abelian groups the situation
is more complicated: an action of an abelian finite group $A$ on a manifold $X$ which is
trivial in cohomology need not satisfy $\chi(X^A)=\chi(X)$ (example: take $A$ to be
the group of $3\times 3$ diagonal matrices with determinant $1$ and entries $\pm 1$, acting
linearly on the unit sphere in $\RR^3$). For much deeper examples see Section
4 in \cite{Pa}.

For any finite group $G$ let $\pi(G)$ denote the set of prime divisors of $|G|$.
Using properties {\it (a)}, {\it (b)} and {\it (c)} and an induction argument we
prove the following.

\begin{theorem}
\label{thm:bounded-primes}
Let $X$ be a connected
manifold without odd cohomology and let $r\in\NN$. There exists an integer
$C(X,r)\geq 1$ such that any finite group $G$ acting effectively on $X$ and satisfying
$|\pi(G)|\leq r$ contains an abelian subgroup $A\subseteq G$ of index at most $C(X,r)$
which can be generated by at most $[\dim X/2]$ elements.
\end{theorem}

See Subsections \ref{ss:review-previous-results} and \ref{ss:induction-scheme}
for an overview of the proof of Theorem \ref{thm:bounded-primes}.

The case $r=1$ in Theorem \ref{thm:bounded-primes} is substantially easier than the
case $r\geq 2$, and it can be proved under the weaker hypothesis that $\chi(X)$ is nonzero
using a fixed point argument (see the comments after Lemma \ref{lemma:one-big-stabiliser}).

The cases $r=1,2$ of Theorem \ref{thm:bounded-primes}
combined with the existence of Hall
subgroups of finite solvable groups imply
easily Theorem \ref{thm:main} for finite solvable groups. Perhaps surprisingly,
the cases $r=1,2$  also allow us to deduce
the full statement of Theorem \ref{thm:main}, at least if one uses the classification
theorem of finite simple groups.
This is based on the following result, proved in \cite{MT}:

\begin{theorem}
\label{thm:MT}
Let $\CCC$ be a collection of finite groups which is closed under taking subgroups
and let $d$, $M$ be positive integers satisfying
the following condition: for any $G \in \CCC$ whose cardinal is divisible
by at most two different primes there exists an abelian
subgroup $A$ of $G$ such that $A$ can be generated by at most $d$ elements and
$[G : A] \le M$.
Then there exists a constant $C_0$ such that any $G\in \CCC$ has
an abelian subgroup of index at most $C_0$ which can be generated by at most $d$
elements.
\end{theorem}

Details of how to deduce
Theorem \ref{thm:main} from Theorem \ref{thm:bounded-primes}
and Theorem \ref{thm:MT} are given in Section \ref{s:proof-thm:main}.

Since an effective action of $G$ on a manifold $X$ is the same as an inclusion of groups
$G\hookrightarrow \Diff(X)$, the statement of Ghys' conjecture can be seen as an
analogue for diffeomorphism groups of the following classic theorem of Jordan.

\begin{theorem}[Jordan, \cite{J}]
\label{thm:Jordan-classic} For any $n\in \NN$ there exists some
constant $C$ such that any finite subgroup of $\GL(n,\CC)$ has an
abelian subgroup of index at most $C$.
\end{theorem}

(See also \cite{BW,CR,M1}.)
The proof of Theorem \ref{thm:bounded-primes} has some remarkable similarities to
Jordan's original proof of Theorem \ref{thm:Jordan-classic},
which is rather different from the usual modern
proofs (a beautiful and very clear exposition
of Jordan's proof is given in \cite{Bu}). In particular, our notion of $K$-stable action
of an abelian group $A$ on a manifold $X$ is closely related to the notion of
$M$-beam in \cite{Bu}.
Also, Lemma \ref{lemma:equacio-diofantina}, which completes the induction step in the proof
of Theorem \ref{thm:bounded-primes}, is exactly the same as Lemma 4.1 in \cite{Bu}.

\subsection{Relation to other work}
For other partial results on Ghys' conjecture see \cite{M1,M2,M3}.
The strategy to prove the main result in
\cite{M2} is very similar to that of
Theorem \ref{thm:bounded-primes} but the setting in \cite{M2} is much simpler,
so it might help
the reader to have a look at \cite{M2} before delving into the details
of the present paper. Two important differences between \cite{M2} and the present
paper are that the notion of $K$-stability is replaced in \cite{M2} by the similar
notion of $C$-rigidity, and that \cite{M2} does not rely on \cite{MT} (so in particular
\cite{M2} does not use the classification of finite simple groups).

In \cite{M3} we prove Ghys' conjecture for all compact manifolds with nonzero Euler
characteristic, for homology spheres, and for all open contractible manifolds.
There is some overlap between the results of the present paper and \cite{M3}, and one
could obtain a complete proof of Theorem \ref{thm:main} avoiding the use of
Theorem \ref{thm:bounded-primes} (which is proved in Section \ref{s:proof-thm:bounded-primes})
and invoking instead the main result in \cite{M3}.
However the arguments in Section \ref{s:proof-thm:bounded-primes}, which are completely
different from those in \cite{M3}, give a much stronger geometric result (independent of CFSG) than the geometric results in \cite{M3}.
Roughly speaking, both here and in \cite{M3} the partial cases of Ghys' conjecture are proved
combining a {\it geometric part}
and an {\it algebraic part}, but the balance between the two parts differ in the two papers.
In concrete terms, while
the {\it algebraic part} in \cite{M3} is the main result in \cite{MT} in its
strongest form, the {\it algebraic part} here is only a weaker form of \cite{MT},
namely Theorem \ref{thm:MT} above.
Conversely, the {\it geometric part} in
\cite{M3} is a much weaker result than the one in the present paper: while in \cite{M3}
we only manage to deal with groups whose cardinal is only divisible by two different
primes, here
we prove a result for an arbitrary number of prime divisors
(Theorem \ref{thm:bounded-primes}).
It is conceivable
that a further development of the ideas in the present paper might lead to a proof
which does not use the CFSG, while this is much more unlikely for
the arguments in \cite{M3}.

Furthermore, the methods in \cite{M3} do not give any control on the Euler characteristic of the fixed point set of abelian subgroups of small index, so they do not allow to prove the analogue
of Corollary \ref{cor:existence-almost-fixed-point} for arbitrary compact manifolds with nonzero Euler characteristic.

Of course if for some manifold there exists an upper bound on the size of the finite
groups which can act effectively on it then this manifold
automatically satisfies Ghys' conjecture.
There exist plenty of examples of such manifolds in the literature: see \cite{P}
for examples of simply connected 6-manifolds with this property and the references
in \cite{P} for more classical constructions of non simply connected examples.

Although Ghys' conjecture originally refers to compact manifolds,
it is natural to wonder to what extent it is true for noncompact connected manifolds.
In \cite{M3} the conjecture is proved for open contractible manifolds.
In contrast, a recent preprint of Popov \cite{Po} gives a beautiful 4-dimensional
example showing that Ghys' conjecture is false for general open connected manifolds.
This example is based on Higman's universal finitely presented group.
A plausible guess seems to be that Ghys' conjecture should true for open connected manifolds
with finitely generated integral cohomology.

\subsection{Conventions}
In this paper by a natural number we understand a strictly positive
integer, so $\NN=\ZZ_{>0}$.
For any prime $p$, $\FF_p$ denotes the field of $p$ elements.
The $p$-part of an abelian group $\Gamma$ is the subgroup $\Gamma_p$
consisting of the elements whose order is a power of $p$.

As already stated,
all manifolds considered in this paper will be, unless we say the contrary,
compact, possibly with boundary, but not necessarily connected.
If $X$ is a manifold, we define the dimension of $X$ to be
$$\dim X:=\max\{\dim X'\mid X'\subseteq X \text{ connected component}\}.$$

All groups will be finite, except from some obvious exceptions such as
positive dimensional Lie groups or diffeomorphism groups.

If a group $G$ acts on a space $X$ we denote by $G_x$
the stabilizer $\{g\in G\mid g\cdot x=x\}$
of any point $x\in X$, and by $X^g$ the
fixed point set $\{x\in X\mid g\cdot x=x\}$
of any $g\in G$. For any subset $S\subseteq G$ we denote by
$X^S$ the intersection $\bigcap_{g\in S}X^g$.

We say that a group $G$ acts on a manifold
$X$ in an ECT way if the action
is effective and the induced action on the integral cohomology
of $X$ is trivial.

\subsection{Contents of the paper}
In Section \ref{s:preliminaries} we collect several basic facts on smooth
actions of finite groups (these are enough, in particular, to prove Ghys'
conjecture for actions of $p$-groups on manifolds with nonzero Euler
characteristic).
Section \ref{s:fixed-point-loci-p-groups} contains results on the cohomology
of the fixed point locus of actions of $p$-groups, proved using localization
in equivariant cohomology. The basic consequence to be used in the subsequent
arguments is an upper bound on the number of connected components of the
fixed point locus.
Section \ref{s:K-stable-actions} introduces the notion of $K$-stable actions
of abelian groups, preceded by some preliminary results.
Section \ref{s:K-stable-actions-no-odd-cohomology}
contains some results on $K$-stable actions of abelian groups on manifolds
without odd cohomology.
Sections \ref{s:proof-thm:bounded-primes}
and \ref{s:proof-thm:main-abelian} contain the proofs of Theorems
\ref{thm:main-abelian} and \ref{thm:bounded-primes} respectively.
Finally, in Section \ref{s:proof-thm:main} we prove Theorem \ref{thm:main}.

\subsection{Acknowledgement}
I am very pleased to acknowledge my indebtedness to Alexandre Turull.
It's thanks to him that Theorem \ref{thm:main},
which in earlier versions of this paper referred only to finite solvable groups, has become
a theorem on arbitrary finite groups.

\section{Preliminaries}

\label{s:preliminaries}

\subsection{Local linearization of smooth finite group actions}

The following result is well known. We recall it because of its crucial role
in some of the arguments of this paper. Statement (1)
implies that the fixed point set of a finite group action on a manifold with boundary
is a neat submanifold in the sense of \S 1.4 in \cite{H}.

\begin{lemma}
\label{lemma:linearization}
Let a finite group $\Gamma$ act smoothly on a manifold $X$, and let $x\in X^{\Gamma}$.
The tangent space $T_xX$ carries a linear action of $\Gamma$, defined as the derivative
at $x$ of the action on $X$, satisfying the following properties.
\begin{enumerate}
\item There exist neighborhoods $U\subset T_xX$ and $V\subset X$, of $0\in T_xX$ and $x\in X$ respectively, such that:
    \begin{enumerate}
    \item if $x\notin\partial X$ then there is a $\Gamma$-equivariant diffeomorphism $\phi\colon U\to V$;
    \item if $x\in\partial X$ then there is $\Gamma$-equivariant diffeomorphism $\phi\colon U\cap \{\xi\geq 0\}\to V$, where $\xi$ is a nonzero
        $\Gamma$-invariant element of $(T_xX)^*$ such that $\Ker\xi=T_x\partial X$.
    \end{enumerate}
\item If the action of $\Gamma$ is effective and $X$ is connected then the action of
$\Gamma$ on $T_xX$ is effective, so it induces an inclusion $\Gamma\hookrightarrow\GL(T_xX)$.
\item \label{item:inclusio-propia-subvarietats-fixes}
If $\Gamma'\vt \Gamma$ and $\dim_xX^{\Gamma}<\dim_xX^{\Gamma'}$ then
there exists an irreducible $\Gamma$-submodule $V\subset T_xX$
on which the action of $\Gamma$ is nontrivial but the action of $\Gamma'$ is trivial.
\end{enumerate}
\end{lemma}
\begin{proof}
We first construct a $\Gamma$-invariant Riemannian metric $g$ on $X$ with respect to which $\partial X\subset X$ is totally geodesic. Take any
tangent vector field on a neighborhood of $\partial X$ whose restriction
to $\partial X$ points inward; averaging over the action of $\Gamma$, we get a
$\Gamma$-invariant vector field which still points inward, and its flow at short time
defines an embedding $\psi:\partial X\times[0,\epsilon)\to X$ for some small $\epsilon>0$
such that $\psi(x,0)=x$ and
$\psi(\gamma\cdot x,t)=\gamma\cdot \psi(x,t)$ for any $x\in \partial X$ and $t\in
[0,\epsilon)$. Let $h$ be a $\Gamma$-invariant Riemannian metric on $\partial X$ and consider any Riemannian metric on $X$ whose restriction to $\psi(\partial X\times[0,\epsilon/2])$
is equal to $h+dt^2$. Averaging this metric over the action of $\Gamma$ we obtain a metric $g$
with the desired property. The exponential map with respect to $g$
gives the local diffeomorphism in (1). To prove (2), assume that the action of
$\Gamma$ on $X$ is effective. (1) implies that for any subgroup
$\Gamma'\subseteq \Gamma$ the fixed point set
$X^{\Gamma'}$ is a submanifold of $X$ and that $\dim_xX=\dim (T_xX)^{\Gamma'}$
for any $x\in X^{\Gamma'}$; furthermore, $X^{\Gamma'}$ is closed by the continuity
of the action. So if some element $\gamma\in\Gamma$
acts trivially on $T_xX$, then by $X^{\gamma}$ is a closed
submanifold of $X$ satisfying $\dim_xX^{\gamma}=\dim X$. Since $X$ is connected this implies $X^{\gamma}=X$,
so $\gamma=1$, because the action of $\Gamma$ on $X$ is effective. Finally, (3)
follows from (1) ($V$ can be defined as any of the irreducible factors in the $\Gamma$-module
given by the perpendicular of $T_xX^{\Gamma}$ in $T_xX^{\Gamma'}$).
\end{proof}

\subsection{Generators of finite abelian subgroups of $\GL(n,\RR)$}
The following result will be used in combination with Jordan's Theorem \ref{thm:Jordan-classic}.

\begin{theorem}
\label{thm:abelian-generat} For any $n\in\NN$ there exists some $C_n\in\NN$ such that
any finite abelian subgroup of $\GL(n,\RR)$ has
a subgroup of index at most $C_n$ which is isomorphic to a
subgroup of $(S^1)^{[n/2]}$ and can be generated by at most $[n/2]$ elements.
\end{theorem}
\begin{proof}
We prove that there exists $C_n\in\NN$ such that any finite abelian subgroup $\Gamma$ of
$\GL(n,\RR)$ has a subgroup of index at most $C_n$ which is isomorphic to
a subgroup of $(S^1)^{[n/2]}$.
We use induction on $n$. The case $n=1$ is obvious, since a finite
subgroup of $\GL(1,\RR)$ has at most $2$ elements. Suppose that
$n>1$ and that the claim has been proved for smaller values of $n$.
Let $\Gamma\subset\GL(n,\RR)$ be a finite and abelian subgroup. By the usual
averaging trick, there is some $\Gamma$-invariant Euclidean norm on $\RR^n$, so replacing
$\Gamma$ by some conjugate we may assume that $\Gamma\subset \O(n,\RR)$.
Since $\Gamma$ is abelian, there exists a common eigenvector $e\in\CC^n$
for all the elements of $\Gamma$. Let $\lambda\colon \Gamma\to
S^1\subset\CC^*$ be defined by $\gamma\cdot e=\lambda(\gamma)e$ for
each $\gamma\in\Gamma$. We distinguish two cases.
If $e=\ov{e}$ then $\lambda(\Gamma)\subseteq\{1,-1\}$ (because $\Gamma$ consists of real
matrices) so $\RR e\subset\RR^n$ is $\Gamma$-invariant, and hence so is $e^{\perp}\subset\RR^n$. Then $[\Gamma:\Ker\lambda]\leq 2$
and the restriction map $\Ker\lambda \to\GL(e^{\perp})$ is injective,
so applying the inductive hypothesis to the action of
$\Ker\lambda $ on $e^{\perp}\subset\RR^n$ we are
done. If $e\neq\ov{e}$ then $\dim_{\RR}(\CC\la e,\ov{e}\ra\cap\RR^n)=2$. Let
$\rho\colon \Gamma\to\GL((\CC\la e,\ov{e}\ra\cap\RR^n)^{\perp})$
be the restriction map. By the inductive hypothesis
there exists a subgroup $R\subseteq \rho(\Gamma)$ of index at most $C_{n-2}$
and an injective map
$\iota\colon R\to (S^1)^{[n/2]-1}$. Then $[\Gamma:\rho^{-1}(R)]=C_{n-2}$ and
$$(\lambda,\iota\circ\rho)\colon \rho^{-1}(R)\to S^1\times
(S^1)^{[n/2]-1}=(S^1)^{[n/2]}$$ is injective because the
restriction of $\rho$ to $\Ker\lambda$ is injective, so the claim is proved.

Now, any finite subgroup $H\subset(S^1)^m$ can be generated by at most $m$ elements.
Indeed, if we denote by $L\subset\RR^m$ the preimage of $H$ via the projection map
$\RR^m\to(\RR/2\pi\ZZ)^m\simeq (S^1)^m$, then $L$ is a discrete subgroup of $\RR^m$, so
it can be generated by at most $m$ elements, see for example \cite[Theorem 5.3.2]{R}.
Since $H$ is a quotient of $L$, the same holds for $H$.
\end{proof}

\subsection{Cohomologically trivial and unipotent actions}

Let $R$ be a commutative ring with unit.
Let a group $G$ act on a topological space $X$. For any $g\in G$ denote by
$g^*\colon H^*(X;R)\to  H^*(X;R)$ be the endomorphism induced by the action of $g$.
We say that the action of  $G$ on $X$ is $R$-cohomologically trivial
($R$-CT, for short) if $g^*$ is the identity for all $g\in G$.
Similarly, we say that the action is $R$-cohomologically unipotent
($R$-CU, for short) if for each $g$ there exists some $m$ such that
$(1-g^*)^m=0$. If $G$ is finite and $R$ is $\ZZ$ or a field of
characteristic zero, then $R$-CU implies $R$-CT; but in general
an action of a finite group $G$ might be $\FF_p$-CU without being
$\FF_p$-CT.

If $R=\ZZ$ then we will write CT instead of $R$-CT. Taking the trivial
lift of the $G$-action to the constant sheaves of abelian groups
on $X$ in the short exact sequence
\begin{equation}
\label{eq:ex-sequence-z-p}
0\to\un{\ZZ}\stackrel{\cdot p}{\longrightarrow}\un{\ZZ}
\to \un{\FF_p}
\to 0
\end{equation}
and applying cohomology, we deduce that, for a given action of $G$,
\begin{equation}
\label{eq:CT-implies-CT-p}
\text{CT}\quad\Longrightarrow\quad \FF_p-\text{CU}\qquad\qquad\text{for any prime $p$.}
\end{equation}
If $H^*(X,\ZZ)$ has no $p$-torsion or is concentrated
in even degrees, as we assume in our main theorem, then CT implies $\FF_p$-CT (however,
in the course of the proof of our main theorem we can not avoid considering
actions on manifolds which can be $\FF_p$-CU and not $\FF_p$-CT, mainly because of Theorem
\ref{thm:accions-unipotents}).

\begin{lemma}
\label{lemma:Minkowski}
For any manifold $X$ without odd cohomology there exists some $C\in\NN$ such that any
finite group $G$ acting on $X$ has a subgroup $G_0\subseteq G$
of index at most $C$ whose action on $X$ is CT.
\end{lemma}
\begin{proof}
Since $X$ is implicitly assumed to be compact,
its cohomology is a finitely generated abelian group,
and since $X$ has no odd cohomology, we have
$H^*(X;\ZZ)\simeq \ZZ^{r}$ for some $r$. So the statement
follows from a result of Minkowski which says that
the size of any finite subgroup of $\GL(r;\ZZ)$ is bounded above by a number
depending only on $r$ (see \cite{Mi,S}, or \cite{M1} for a proof using
Jordan's theorem).
\end{proof}

For a generalization of the preceding lemma, see Lemma 2.2 in \cite{M2}.

Recall that a group $G$ is said to act on a manifold $X$ in an ECT way
if the action is effective and CT.

\subsection{Good actions and good triangulations}
All simplicial complexes we consider in this paper are implicitly assumed to be finite.
Let $G$ be a 
group and let $\cC$ be a simplicial complex
endowed with an action of $G$. We say that this action is good  if for
any $g\in G$ and any $\sigma\in \cC$ such that $g(\sigma)=\sigma$
we have $g(\sigma')=\sigma'$ for any subsimplex $\sigma'\subseteq\sigma$
(equivalently, the restriction of the action of $g$ to $|\sigma|\subset|\cC|$ is the identity).
This property is called condition (A) in Chapter III, \S 1, of \cite{Br}.
If $\cC$ is a simplicial complex and $G$ acts on
$\cC$, then the induced action of $G$ on the barycentric subdivision
$\sd \cC$ is good (see Proposition 1.1 of Chapter III in \cite{Br}).

Suppose that $G$ acts on a manifold $Y$. A $G$-good triangulation of $Y$ is a
pair $(\cC,\phi)$, where $\cC$ is a simplicial complex endowed with a
good action of $G$ and $\phi\colon Y\to |\cC|$ is a $G$-equivariant homeomorphism.
For any smooth action of a finite group $G$ on a manifold $X$ there exist
$G$-good triangulations of $X$ (by the previous comments it suffices to prove
the existence of a $G$-equivariant triangulation; this can be easily obtained
adapting the construction of triangulations of smooth manifolds given in \cite{C}
to the finitely equivariant setting; for much more detailed results, see \cite{I}).

\subsection{Euler characteristic of fixed point set for cyclic groups}

\begin{lemma}
\label{lemma:Euler-fixed-point-set-cyclic} Let $\Gamma$ be a cyclic
group, let $k$ be a field of characteristic $p$ ($p$ can be $0$),
and assume that $\Gamma$ acts on a 
manifold $Y$ in a $k$-CU way. We have
$$\chi(Y^{\Gamma})\equiv \chi(Y) \mod p.$$
\end{lemma}
\begin{proof}
Let $(\cC,\phi)$ be a $\Gamma$-good triangulation of $Y$. Let
$S_i(\cC)$ denote the vector space of $i$-simplicial chains (of ordered
simplices) with coefficients in $k$, let $Z_i(\cC)\subseteq S_i(\cC)$
denote the cycles and let $B_i(\cC)\subseteq Z_i(\cC)$ denote the
boundaries.
Let $\gamma\in \Gamma$ be a generator. Since $k$ has characteristic
$p$ and the action of $\Gamma$ on $\cC$ is good, we have
$$\chi(Y^{\Gamma}) = \sum_i (-1)^i\dim_k S_i(\cC^{\Gamma})
\equiv \sum_i (-1)^i\Tr(\gamma\colon S_i(\cC)\to S_i(\cC)) \mod p.$$
Using the exact sequences
$$0\to B_i(\cC)\to Z_i(\cC)\to H_i(\cC)\to 0,
\qquad
0\to Z_i(\cC)\to S_i(\cC)\to B_{i-1}(\cC)\to 0$$
and the fact that the trace is additive on short exact sequences, we
deduce
$$\sum_i (-1)^i\Tr(\gamma\colon S_i(\cC)\to S_i(\cC))=
\sum_i (-1)^i\Tr(\gamma\colon H_i(\cC)\to H_i(\cC)).$$
Finally, since the action is $k$-CU, we have
$$\sum_i (-1)^i\Tr(\gamma\colon H_i(\cC)\to H_i(\cC))\equiv \chi(Y) \mod p.$$
Putting together the previous equalities we obtain the desired result.
\end{proof}

\subsection{Points with big stabilizer for actions of $p$-groups}

\begin{lemma}
\label{lemma:one-big-stabiliser}
Let $Y$ be a 
manifold satisfying $\chi(Y)\neq 0$. Let $p$ be a prime,
and let $G$ be a $p$-group acting smoothly on $Y$. Let $r$ be
the biggest nonnegative integer such that $p^r$ divides $\chi(X)$.
There exists some $y\in Y$ whose stabilizer has index at most $p^r$.
\end{lemma}
\begin{proof}
Let $(\cC,\phi)$ be a $\Gamma$-good triangulation of $Y$.
The cardinal of each of the orbits of $G$ acting on $\cC$ is a power of $p$.
If the cardinal of all orbits were divisible by $p^{r+1}$, then
for each $d$ the cardinal of the set of $d$-dimensional simplices in $\cC$
would be divisible by $p^{r+1}$, and consequently
$\chi(Y)=\chi(\cC)$ would also be divisible by $p^{r+1}$, contradicting
the definition of $r$. Hence, there must be at least one simplex $\sigma\in \cC$
whose orbit has at most $p^r$ elements. This means that the stabilizer $G_{\sigma}$ of
$\sigma$ has index at most $p^r$. If $y\in Y$ is a point such that
$\phi(y)\in |\sigma|\subseteq |\cC|$,
then $y$ is fixed by $G_{\sigma}$, because the triangulation is good.
\end{proof}

Note the combination of Lemma \ref{lemma:linearization} with Jordan's Theorem
\ref{thm:Jordan-classic} and Lemma \ref{lemma:one-big-stabiliser} implies
Ghys' conjecture for actions of finite $p$-groups on manifolds
with nonzero Euler characteristic.

\section{Fixed point loci of actions of abelian $p$-groups}
\label{s:fixed-point-loci-p-groups}

\subsection{Betti numbers of fixed point sets}

For any group $G$ we denote by $Z(G)$ the center of $G$.

\begin{theorem}
\label{thm:accions-unipotents}
Let $p$ be a prime number, and let $G$ be a $p$-group. Let
$\Gamma\subseteq Z(G)$ be a subgroup isomorphic to $\ZZ/p\ZZ$. Suppose that $G$ acts
on a 
manifold $Y$ and that the action is $\FF_p$-CU.
Then:
\begin{enumerate}
\item $\sum_j b_j(Y^\Gamma;\FF_p)\leq\sum_j b_j(Y;\FF_p)$,
\item the natural action of $G/\Gamma$ on $Y^{\Gamma}$ is $\FF_p$-CU,
\item $\chi(Y^\Gamma)\equiv\chi(Y)\mod p$.
\end{enumerate}
\end{theorem}
\begin{proof}
All cohomology groups in this proof will be with coefficients in $\FF_p$.
The Leray--Serre spectral sequence for the Borel construction
$Y_\Gamma\to B\Gamma$ has second page $E_2^{ab}=H^a(B\Gamma;\hH^b(Y))$, where
$\hH^b(Y)\to B\Gamma$ is the local system $E\Gamma\times_\Gamma H^b(Y)$.
We claim that
$\dim_{\FF_p}E^{ab}_2\leq \dim_{\FF_p}H^b(Y)$.
Let $\gamma\in \Gamma$ be a generator, and let $\gamma^*\colon H^b(Y)\to H^b(Y)$
be the induced morphism. Define $F^j:=\Ker(1-\gamma^*)^j\subseteq H^b(Y)$.
Since $\gamma^*$ commutes with $1-\gamma^*$, $\gamma^*(F^j)\subseteq F^j$. Furthermore,
$(1-\gamma^*)F^j\subseteq F^{j-1}$, so $\gamma^*$ acts as the identity on $F^j/F^{j-1}$.
Let $\fF^j=B\Gamma\times_\Gamma F^j\to B\Gamma$. Let $r\in\NN$ be the smallest number such
that $(1-\gamma^*)^r=0$. Then
$\fF^1\subseteq\fF^2\subseteq\dots\subseteq\fF^r=\hH^b(Y)$ is a filtration
of local systems of $\FF_p$-vector spaces. Taking cohomology on the short
exact sequence of local systems
$0\to\fF^{j-1}\to\fF^j\to \fF^j/\fF^{j-1}\to 0$
we get
$$\dots\to H^a(B\Gamma;\fF^{j-1})\to H^a(B\Gamma;\fF^{j})\to
H^a(B\Gamma;\fF^j/\fF^{j-1})\to\dots,$$
which implies
$\dim H^a(B\Gamma;\fF^{j})\leq \dim H^a(B\Gamma;\fF^{j-1})+\dim H^a(B\Gamma;\fF^j/\fF^{j-1}).$
Since $\gamma^*$ acts as the identity on $F^j/F^{j-1}$, the local system $\fF^j/\fF^{j-1}$
is globally trivial with fiber $F^j/F^{j-1}$, so
$H^a(B\Gamma;\fF^j/\fF^{j-1})\simeq H^a(B\Gamma)\otimes_{\FF_p}F^j/F^{j-1}\simeq F^j/F^{j-1}$
(here we use that $H^j(B\Gamma)\simeq\FF_p$ for every $j\geq 0$). Hence
$$\dim H^a(B\Gamma;\fF^{j})\leq \dim H^a(B\Gamma;\fF^{j-1})+\dim F^j/F^{j-1}.$$
Summing over $j$ we get $\dim H^a(B\Gamma;\hH^b(Y))\leq\dim H^b(Y)$, so the
claim is proved.

Since the
Leray--Serre spectral sequence converges to $H^*_\Gamma(Y)$,
the claim implies that
$$\dim H^k_\Gamma(Y)\leq \sum_{j=0}^k \dim H^j(Y).$$
The inclusion of the fixed point set $Y^\Gamma\to Y$ induces, for any $k>\dim Y$,
an isomorphism $H^k_\Gamma(Y)\to H^k_\Gamma(Y^\Gamma)$
(this is standard: apply Mayer--Vietoris for
the equivariant cohomology of $Y$ to the covering given by a tubular neighborhood of $Y^\Gamma$
and $Y\setminus Y^\Gamma$; then use that the equivariant cohomology of a free action coincides
with the usual cohomology of the quotient, and finally take into account that the cohomology of
a manifold vanishes in degrees bigger than the dimension).
But $H^*_\Gamma(Y^\Gamma)\simeq H^*(B\Gamma)\otimes_{\FF_p}H^*(Y^\Gamma)$,
so $\dim H^k_\Gamma(Y^\Gamma)=\sum_{j=0}^k\dim H^j(Y^\Gamma)$. Combining this with the previous
inequality, and using the fact that the cohomologies of $Y^\Gamma$ and $Y$ vanish in
dimensions $>\dim Y$, we get
$$\sum_j \dim H^j(Y^\Gamma) \leq \sum_j \dim H^j(Y),$$
so (1) is proved.

We now prove (2). Since $\Gamma\subseteq Z(G)$, the action of $G$ on $Y$ extends to an action
on the Borel construction $Y_\Gamma$ lifting the trivial action on $B\Gamma$. It follows that
$G$ acts naturally on the Leray--Serre spectral sequence for $Y_\Gamma\to B\Gamma$.
Let $g\in G$ be any element
and let $g^*\colon H^*(Y)\to H^*(Y)$ be the map induced by its action on $Y$.
By assumption, there exists some $r$ such that $(1-g^*)^r\colon H^*(Y)\to H^*(Y)$
is the zero map.
This implies that $(1-g^*)^r$ is the zero map on $H^a(B\Gamma;\hH^b(Y))$, and by the naturality
of the spectral sequence it is also the zero map on $E^{ab}_{\infty}$. By convergence, for any
$k$ there is a natural filtration $L^k_k\subseteq L^k_1\subseteq\dots\subseteq
 L^k_k=H^k_\Gamma(Y)$
and a natural isomorphism $L^k_i/L^k_{i+1}\simeq E^{i,k-i}_{\infty}$. It follows that
$(1-g^*)^r$ is the zero map on $\bigoplus_i L^k_i/L^k_{i+1}$, which implies that
$(1-g^*)^{rk}$ is the zero map on $H^k_\Gamma(Y)$.

Let $k:=\dim Y+1$.
Since $\Gamma\subseteq Z(G)$, the fixed point set $Y^\Gamma$ is $G$ invariant,
so the inclusion $Y^\Gamma\to Y$ is a $G$-equivariant map. Hence the isomorphism
$H^k_\Gamma(Y)\to H^k_\Gamma(Y^\Gamma)$ is $G$-equivariant. It follows that
$(1-g^*)^{rk}$ is the zero map on $H^k_\Gamma(Y)$. Using the fact that the isomorphism
$H^*_\Gamma(Y^\Gamma)\simeq H^*(B\Gamma)\otimes_{\FF_p}H^*(Y^\Gamma)$
is $G$-equivariant and that $G$ acts
trivially on $H^*(B\Gamma)$ (because $\Gamma\subseteq Z(G)$),
we deduce that $(1-g^*)^{rk}$ is the zero
map on $H^k_\Gamma(Y^\Gamma)\simeq \bigoplus_j H^j(Y^\Gamma)$. Hence (2) is proved.

Finally, (3) follows from Lemma \ref{lemma:Euler-fixed-point-set-cyclic}.
\end{proof}

\begin{corollary}
\label{cor:cohom-no-augmenta}
Let $X$ be a manifold and let $p$ be a prime.
Let $G$ be a $p$-group acting on $X$ in a $\FF_p$-CU way.
We have
$$\sum_j b_j(X^G;\FF_p)\leq \sum_j b_j(X;\FF_p)$$
and $\chi(X^G)\equiv \chi(X)\mod p$.
\end{corollary}
\begin{proof}
Take a maximal filtration
$1=G_0\subsetneq G_1\subsetneq G_2\subsetneq \dots \subsetneq G_k=G$
such that $G_j/G_{j-1}$ is central in $G/G_{j-1}$ for each $j$ (such filtrations
exist because $p$-groups are nilpotent), and
apply recursively Lemma \ref{thm:accions-unipotents} to
$G_j/G_{j-1}\subset G/G_{j-1}$ acting on $X^{G_{j-1}}$,
beginning with $j=1$ and increasing $j$ one unit at each step.
\end{proof}

\section{$K$-stable actions of abelian groups}
\label{s:K-stable-actions}

Recall that if $X$ is a non necessarily connected manifold we call the
dimension of $X$ (denoted by $\dim X$) the maximum of the dimensions of the connected components
of $X$.

\subsection{Preliminaries}

\begin{lemma}
\label{lemma:chains-inclusions}
Given $m,k\in\NN$ there exists $C_M(m,k)\in\NN$ with the following property.
If $X$ is a manifold of dimension $m$,
$X_1\subsetneq X_2\subsetneq\dots\subsetneq X_r\subsetneq X$
are strict inclusions of neat\footnote{See \S 1.4 in \cite{H}.}
submanifolds, and each $X_i$ has
at most $k$ connected components, then $r\leq C_M(m,k)$.
\end{lemma}
\begin{proof}
Let $X_1\subsetneq X_2\subsetneq\dots\subsetneq X_r\subsetneq X$ be as in the statement of the lemma.
For any $i$ let $d(X_i)=(d_{m-1}(X_i),\dots,d_{0}(X_i))\in\ZZ_{\geq 0}^m$, where
$d_{j}(X_i)$ denotes the number of connected components of $X_i$ of dimension $j$.
Each $X_i$ has at most $k$ connected components, so $d(X_i)$ belongs to
$\dD=\{(d_1,\dots,d_m)\in\ZZ_{\geq 0}^m\mid \sum d_j\leq k\}$.
Consider the lexicographic order on $\ZZ_{\geq 0}^m$, so that
$(a_1,\dots,a_m)>(b_1,\dots,b_m)$ if there is some $1\leq l\leq m$ such that
$a_j=b_j$ for any $j\leq l-1$ and $a_l>b_l$. We claim that for each $i$ we have
$d(X_i)>d(X_{i-1})$.
To prove the claim, let us denote by
$X_{i-1,1},\dots,X_{i-1,r}$ (resp. $X_{i,1},\dots,X_{i,s}$) the connected components of
$X_{i-1}$ (resp. $X_i$), labelled in such a way that
$\dim X_{i-1,j-1}\geq \dim X_{i-1,j}$ and $\dim X_{i,j-1}\geq \dim X_{i,j}$
for each $j$.
Since $X_{i-1}\subset X_i$, there exists a map $f\colon \{1,\dots,r\}\to\{1,\dots,s\}$
such that $X_{i-1,j}\subset X_{i,f(j)}$, which implies that $\dim X_{i-1,j}\leq\dim X_{i,f(j)}$.
Let $J$ be the set of indices $j$ such that $\dim X_{i-1,j}<\dim X_{i,f(j)}$.
We distinguish two cases.
\begin{itemize}
\item {\bf Case 1.}
If $J=\emptyset$, so that
$\dim X_{i-1,j}=\dim X_{i,f(j)}$ for each $j$, then $X_{i-1}\neq X_i$ implies that
$X_i=X_{i-1}\sqcup X_i'$ for some nonempty and possibly disconnected $X_i'\subset X$,
because by assumption $X_{i-1}$ is a neat submanifold of $X$.
This implies that $d_{\delta}(X_i)\geq d_{\delta}(X_{i-1})$ for each $\delta$, and
the inequality is strict for at least one $\delta$. Hence $d(X_i)>d(X_{i-1})$.
\item {\bf Case 2.} Suppose that $J\neq\emptyset$.
Let $l=\dim X_{i,\min f(J)}$. If $l+1\leq\delta\leq m-1$ then any
$\delta$-dimensional connected component of $X_{i-1}$ is also a connected
component of $X_i$, so
$d_{\delta}(X_i)\geq d_{\delta}(X_{i-1})$
(this is not necessarily an equality, since there might be some $\delta$-dimensional
connected component of $X_i$ which does not contain any connected component of $X_{i-1}$),
whereas $d_l(X_i)>d_l(X_{i-1})$. This implies again that $d(X_i)>d(X_{i-1})$, so the proof
of the claim is complete.
\end{itemize}
The claim implies that $r\leq |\dD|$, so letting $C_M(m,k):=|\dD|$ the
lemma is proved.
\end{proof}

\begin{lemma}
\label{lemma:isotropy-subgroups}
Let $X$ be a 
manifold and
let $r\in\NN$. There exists an integer $C_S(X,r)$ with the following property.
Let a $p$-group $\Gamma$ act in a $\FF_p$-CU way on $X$ and let
$$\sS_r(X)=\{\Gamma'\subseteq  \Gamma\mid \exists x\in X\text{ such
that }\Gamma'=\Gamma_x,\text{ and }[\Gamma:\Gamma']\leq p^r\}.$$
Then $|\sS_r(X)|\leq C_S(X,r)$.
\end{lemma}
\begin{proof}
Let $X$ and $\Gamma$ be as in the statement of the lemma.
By Corollary \ref{cor:cohom-no-augmenta}, for any subgroup $\Gamma'\subseteq \Gamma$
we have
\begin{equation}
\label{eq:def-k-p-X}
b_0(X^{\Gamma'};\FF_p)\leq \sum_j b_j(X^{\Gamma'};\FF_p)\leq
\sum_j b_j(X;\FF_p)\leq k(X):=\sum_j b_j^{\infty}(X),
\end{equation}
where $b_j^{\infty}(X)=\max_p\{b_j(X;\FF_p)\}$.
Hence $X^{\Gamma'}$ has at most $k(X)$ connected components.
Let $C_M=C_M(\dim X,k(X))$ be the number defined by Lemma \ref{lemma:chains-inclusions}.
We claim that
$$C_S(X,r):=C_M+2^{p^{rC_M}}$$ is an upper bound on $|\sS_r(X)|$.
Indeed, suppose that $|\sS_r(X)|>C_S(X,r)$. We are going to construct
a sequence of elements $\Gamma_0,\Gamma_1,\dots,\Gamma_{C_M}\in\sS_r(X)$ such that, defining
$\Gamma(i):=\Gamma_0\cap\dots\cap \Gamma_i$, we have $X^{\Gamma(i-1)}\subsetneq X^{\Gamma(i)}$
for each $i\geq 1$. This will contradict Lemma \ref{lemma:chains-inclusions}
(indeed, by (1) in Lemma \ref{lemma:linearization} each $X^{\Gamma(i)}$ is a neat
submanifold of $X$), so the proof that $|\sS_r(X)|\leq C_S(X,r)$ will be complete.

We construct the sequence $\Gamma_j$ using induction on $j$.
Choose $\Gamma_0$ to be any element in $\sS_r(X)$.
Let $1\leq j\leq C_M$,
and assume that $\Gamma_0,\dots,\Gamma_{j-1}\in\sS_r(X)$ have been chosen.
Then $\Gamma(j-1)$ has index at most $p^{rj}$ in $\Gamma$, so the set
$\Gamma/\Gamma(j-1)$ has at most $p^{rj}\leq p^{rC_M}$ elements. This implies that
the number of subgroups of $\Gamma$ which contain $\Gamma(j-1)$ is at most
$2^{p^{rC_M}}$, because each subgroup $\Gamma'\subseteq\Gamma$
containing $\Gamma(j-1)$ projects to
a subset of $\Gamma/\Gamma(j-1)$, and $\Gamma'$ can be recovered from this subset.
Since we are assuming that $\sS_r(X)$ has more than $C_M+2^{p^{rC_M}}$ elements,
there must exist some $\Gamma_j\in\sS_r(X)$ which is neither any of
the subgroups $\Gamma_0,\dots,\Gamma_j$ nor a subgroup of $\Gamma$ containing $\Gamma(j-1)$.
To finish the argument, we prove that $X^{\Gamma(j-1)}\subsetneq X^{\Gamma(j)}$.
Since $\Gamma_j\in\sS_r(X)$, there exists some $x\in X$ such that $\Gamma_x=\Gamma_j$.
And since $\Gamma(j-1)$ is not contained in $\Gamma_j$, $x$ does not belong to
$X^{\Gamma(j-1)}$ (indeed, if $x$ were contained in $X^{\Gamma(j-1)}$ then we would
have $\Gamma(j-1)\subseteq\Gamma_x$). But since $\Gamma(j)\subseteq \Gamma_j$, we have
$x\in X^{\Gamma(j)}$. In conclusion, $x\in X^{\Gamma(j)}\setminus X^{\Gamma(j-1)}$,
so $X^{\Gamma(j-1)}\subsetneq X^{\Gamma(j)}$.
\end{proof}

\begin{lemma}
\label{lemma:exists-Gamma-chi-p}
Let $X$ be a 
manifold, and let $p$ be any prime
number. There exists a number $C_{p,\chi}\in\NN$, depending on $X$ and $p$, with the
following property. For any $p$-group $\Gamma$ acting on $X$ in a $\FF_p$-CU way, there
exists a subgroup $\Gamma_{\chi}\subseteq \Gamma$ of index at most $C_{p,\chi}$ such that
for any $\Gamma_0\subseteq \Gamma_{\chi}$ we have
$\chi(X^{\Gamma_0})=\chi(X).$
Furthermore, there exists a constant $P_{\chi}$, depending only on $X$, such that
if $p\geq P_{\chi}$ then $C_{p,\chi}=1$.
\end{lemma}
\begin{proof}
Let $r\in\ZZ_{\geq 0}$ be the smallest number such that $|\chi(X)+ap^{r+1}|>\sum_jb_j(X;\FF_p)$
for any nonzero integer $a$. We are going
to prove that $C_{p,\chi}:=p^{rC_S(X,r)}$, where $C_S(X,r)$ is defined in Lemma
\ref{lemma:isotropy-subgroups}, does the job.
Let $\Gamma$ be a $p$-group acting on $X$, and define
$$\Gamma_{\chi}:=\bigcap_{\Gamma'\in\sS_r(\Gamma)}\Gamma'.$$
By Lemma \ref{lemma:isotropy-subgroups}, $[\Gamma:\Gamma_{\chi}]\leq C_{p,\chi}$.
We now prove that if $\Gamma_0\subseteq \Gamma_{\chi}$ then $\chi(X^{\Gamma_0})=\chi(X)$.
Consider a $\Gamma$-good triangulation $(\cC,\phi)$ of $X$.
We have $|\cC|^{\Gamma_0}=|\cC^{\Gamma_0}|$, so
\begin{equation}
\label{eq:chi-comptar}
\chi(X)-\chi(X^{\Gamma_0})=\chi(\cC)-\chi(\cC^{\Gamma_0})=
\sum_{j\geq 0}(-1)^j\sharp\{\sigma\in \cC\setminus \cC^{\Gamma_0}\mid \dim\sigma=j\}.
\end{equation}
If $\sigma\in \cC\setminus \cC^{\Gamma_0}$ then the stabilizer $\Gamma_{\sigma}=\{\gamma\in\Gamma\mid\gamma\cdot\sigma=\sigma\}$ does not contain $\Gamma_0$, and this implies that $\Gamma_{\sigma}\notin\sS_r(X)$.
We thus have $[\Gamma:\Gamma_{\sigma}]\geq p^{r+1}$, because otherwise $\Gamma_{\sigma}$ would belong to $\sS_r(X)$ (indeed, $\Gamma_{\sigma}$ is the stabilizer of the baricenter
of $\sigma$). This implies that the cardinal of the orbit $\Gamma\cdot\sigma$
is divisible by $p^{r+1}$. Repeating this argument for all
$\sigma\in \cC\setminus \cC^{\Gamma_0}$ and using (\ref{eq:chi-comptar}), we conclude that $\chi(X)-\chi(X^{\Gamma_0})$ is divisible by $p^{r+1}$.
Now, we have
$|\chi(X^{\Gamma_0})|\leq \sum_j b_j(X^{\Gamma_0};\FF_p)
\leq \sum_j b_j(X;\FF_p)$
(the first inequality is obvious, and the second one follows from Corollary \ref{cor:cohom-no-augmenta} using the hypothesis that the action of $\Gamma$
on $X$ is $\FF_p$-CU).
By our choice of $r$, the congruence $\chi(X^{\Gamma_0})\equiv\chi(X)\mod p^{r+1}$
and the inequality
$|\chi(X^{\Gamma_0})|\leq \sum_j b_j(X;\FF_p)$
imply that $\chi(X^{\Gamma_0})=\chi(X)$.

We now prove the last statement. Since $X$ is implicitly assumed to be
compact, its cohomology is finitely generated,
so in particular the torsion of its integral cohomology is bounded. Hence there exists some $p_0$
such that if $p\geq p_0$ then $b_j(X;\FF_p)=b_j(X)$ for every $j$. Define $P_{\chi}:=\max\{p_0,2\sum_jb_j(X)\}$. If $p\geq P_{\chi}$ then the number $r$ defined
at the beginning of the proof is equal to $0$, and this implies that $C_{p,\chi}=1$.
\end{proof}

\subsection{$K$-stable actions: abelian $p$-groups}

Let $\Gamma$ be an abelian $p$-group acting on a manifold $X$.
Recall that for any $x\in X^{\Gamma}$ the space
$T_xX/T_x^{\Gamma}X$ carries a linear action of $\Gamma$,
induced by the derivative at $x$ of the action on $X$,
and depending up to isomorphism only on the connected component of $X^{\Gamma}$ to which
$x$ belongs (observe that $T_xX/T_x^{\Gamma}X$ is the fiber over $x$ of the normal
bundle of the inclusion of $X^{\Gamma}$ in $X$) .

Let $K\geq 1$ be a real number.
We say that the action
of $\Gamma$ on $X$ is {\bf $K$-stable} if:
\begin{enumerate}
\item the action of $\Gamma$ on $X$ is $\FF_p$-CU, and $\chi(X^{\Gamma_0})=\chi(X)$ for any subgroup $\Gamma_0\subseteq \Gamma$;
\item for any $x\in X^{\Gamma}$ and any
character $\rho\colon \Gamma\to\CC^*$ occurring in the $\Gamma$-module
$T_xX/T_xX^{\Gamma}$ we have
$$[\Gamma:\Ker\rho]\geq K.$$
\end{enumerate}

(This definition is equivalent to the one given in the introduction thanks to (3) in Lemma \ref{lemma:linearization}.)
Note that if $\Gamma$ acts trivially on $X$ then the action is $K$-stable for any $K$.
If $\Gamma\neq\{1\}$ acts effectively and the action is $K$-stable,
then we can  estimate $|\Gamma|\geq K$.

When the manifold $X$ and the action of $\Gamma$ on $X$ are clear from the context,
we will sometimes abusively say that $\Gamma$ is $K$-stable.
For example, if a group $G$ acts
on $X$ we will say that an abelian $p$-subgroup $\Gamma\subseteq G$ is $K$-stable if
the restriction of the action of $G$ to $\Gamma$ is $K$-stable.

\begin{lemma}
\label{lemma:exists-Gamma-stable-p}
There exists a number $C_{p,K}\geq 1$, depending on $X$, $p$ and $K$, with
the following property.
Let $\Gamma$ be an abelian $p$-group acting on $X$ in a $\FF_p$-CU way and
so that for any subgroup $\Gamma'\subseteq \Gamma$ we have
$\chi(X^{\Gamma'})=\chi(X)$. Then $\Gamma$
has a $K$-stable subgroup of index at most $C_{p,K}$.
Furthermore, if $p\geq K$, then $C_{p,K}=1$.
\end{lemma}
\begin{proof}
Let $C_M=C_M(\dim X,\sum_j b_j(X;\FF_p))$
be the number defined by Lemma \ref{lemma:chains-inclusions}.
Let $K_-$ be equal to $\max\{K-1,1\}$ if $K$ is an integer, or $[K]$ if $K$ is not an integer.
We claim that if
$p\leq \chi(X)\dim X$ then $C_{p,K}:=K_-^{C_M-1}$ has the
desired property.

Let $\Gamma_0\subseteq\Gamma$ be a subgroup.
If $\Gamma_0$ is $K$-stable, we simply define $\Gamma_{K,\stable}:=\Gamma_0$
and we are done. If $\Gamma_0$ is not $K$-stable, then
there exists some $x\in X^{\Gamma_0}$ and a
character $\rho\colon \Gamma_0\to\CC^*$ occurring in the $\Gamma_0$-module
$T_xX/T_xX^{\Gamma_0}$ such that
$[\Gamma_0:\Ker\rho]\leq K-1$. Choose one such $x$ and $\rho$
and define $\Gamma_1:=\Ker\rho\subsetneq\Gamma_0$. If $\Gamma_1$ is $K$-stable, then we define
$\Gamma_{K,\stable}:=\Gamma_1$ and we stop, otherwise
we repeat the same procedure with $\Gamma_0$ replaced by $\Gamma_1$ and define
a subgroup $\Gamma_2\subsetneq\Gamma_1$. And so on.
Each time we repeat this procedure, we
go from one group $\Gamma_i$ to a subgroup $\Gamma_{i+1}$ satisfying $X^{\Gamma_i}\subsetneq X^{\Gamma_{i+1}}$. We claim that this process must terminate after
at most $C_M-1$ steps, i.e.,
if the groups $\Gamma_0,\Gamma_1,\dots,\Gamma_{C_M-2}$ are not
$K$-stable, then $\Gamma_{C_M-1}$ is $K$-stable. For otherwise, applying yet once more
the procedure, we would define $\Gamma_{C_M}\subsetneq \Gamma_{C_M-1}$, and we would
obtain a chain of strict inclusions of submanifolds of $X$
$$X^{\Gamma_0}\subsetneq X^{\Gamma_1}\subsetneq\dots\subsetneq X^{\Gamma_{C_M}}.$$
Now, for each $j$ the manifold $X^{\Gamma_j}$ is a neat submanifold
of $X$ (by (1) in Lemma \ref{lemma:linearization}), and
the number of connected components of $X^{\Gamma_j}$ satisfies
(by Corollary \ref{cor:cohom-no-augmenta})
$$|\pi_0(X^{\Gamma_j})|=b_0(X^{\Gamma_j};\FF_p)\leq
\sum_j b_j(X^{\Gamma_j};\FF_p)\leq \sum_j b_j(X;\FF_p),$$
so our assumption leads to a contradiction with Lemma \ref{lemma:chains-inclusions}.
Since for each $i$ such that $\Gamma_{i+1}\subseteq\Gamma_i$ is defined we have
$[\Gamma_i:\Gamma_{i+1}]\leq K_-$, it follows that $\Gamma$ must have
a $K$-stable subgroup $\Gamma_{\stable}$
satisfying $[\Gamma:\Gamma_{\stable}]\leq K_-^{C_M-1}$.

Finally, if $p\geq K$ then the we can take $C_{p,K}=1$,
i.e., any subgroup $\Gamma'\subseteq \Gamma$ is $K$-stable. Indeed, property (1)
in the definition of $K$-stability
is satisfied by hypothesis, and property (2) holds because any
strict subgroup of a $p$-group has index at least $p$.
\end{proof}

If $\Gamma_0,\Gamma_1$
are subgroups of an abelian $p$-group $\Gamma$,
we denote by $\Gamma_0\Gamma_1\subseteq\Gamma$ the group consisting
of the elements $\gamma_0\gamma_1$ for each $\gamma_j\in\Gamma_j$, $j=1,2$.
We have
$$X^{\Gamma_0\Gamma_1}=X^{\Gamma_0}\cap X^{\Gamma_1}.$$

\begin{lemma}
\label{lemma:mixing-stable}
Let $\Gamma$ be an abelian $p$-group acting on $X$ in a $\FF_p$-CU way and
with the property
that for any subgroup $\Gamma'\subseteq \Gamma$ we have
$\chi(X^{\Gamma'})=\chi(X)$. Assume that
$\Gamma_0,\Gamma_1\subseteq \Gamma$ are $K$-stable subgroups. Then
$\Gamma_0\Gamma_1$ is a $K$-stable subgroup.
\end{lemma}
\begin{proof}
By hypothesis, $\chi(X^{\Gamma_0\Gamma_1})=\chi(X)$, so we only need
to check that the action of $\Gamma_0\Gamma_1$ satisfies condition (2)
in the definition of $K$-stability. Let $x\in X^{\Gamma_0\Gamma_1}$.
The characters of $\Gamma_0\Gamma_1$ occurring in the
$\Gamma_0\Gamma_1$-module $T_xX/T_xX^{\Gamma_0\Gamma_1}$ are precisely
the nontrivial characters of $\Gamma_0\Gamma_1$ occurring on $T_xX$. If
$\theta\colon \Gamma_0\Gamma_1\to\CC^*$ is one such character, then
its restriction to $\Gamma_j$, for at least one $j$, has to be
nontrivial. Hence $\theta|_{\Gamma_j}$ is one of the characters
occurring in the $\Gamma_j$-module $T_xX/T_xX^{\Gamma_j}$.
But we have
$[\Gamma_0\Gamma_1:\Ker\theta]\geq [\Gamma_j:\Ker\theta\cap\Gamma_j]
\geq K,$
the last inequality following from the assumption that $\Gamma_j$
is $K$-stable. So $\Gamma_0\Gamma_1$ is $K$-stable.
\end{proof}

\subsection{Fixed point sets and inclusions of groups}
Let $X$ be a connected manifold.
If $A,B$ are subspaces of $X$, we will write
$$A\prec B$$
whenever $A\subseteq B$ and each connected component of $A$ is a connected component of $B$.

\begin{lemma}
\label{lemma:subgroup-stable-general}
Let $\Gamma$ be an abelian $p$-group acting on $X$ in a $K$-stable way.
If a subgroup $\Gamma_0\subseteq \Gamma$ has index smaller than $K$, then $X^{\Gamma}\prec X^{\Gamma_0}$.
\end{lemma}
\begin{proof}
We clearly have $X^{\Gamma}\subseteq X^{\Gamma_0}$, so it suffices to prove that
for each $x\in X^{\Gamma}$ we have $\dim_xX^{\Gamma}=\dim_xX^{\Gamma_0}$.
If this is not the case for some $x\in X^{\Gamma}$ then, by
(\ref{item:inclusio-propia-subvarietats-fixes}) in
Lemma \ref{lemma:linearization}, there exist an irreducible $\Gamma$-submodule
of $T_xX/T_xX^{\Gamma}$ on which the action of $\Gamma_0$ is trivial.
Let $\rho\colon \Gamma\to\CC^*$
be the character associated to this submodule. Then $\Gamma_0\subseteq \Ker\rho$, which implies
that $[\Gamma:\Ker\rho]<K$, contradicting the hypothesis that $\Gamma$ is $K$-stable.
\end{proof}

\begin{lemma}
\label{lemma:existeix-element-generic-p}
Suppose that $K>(\dim X)(\sum_jb_j(X;\FF_p))$, and
let $\Gamma$ be an abelian $p$-group acting on $X$ in a $K$-stable way.
There exists an element $\gamma\in\Gamma$ such that
$X^{\Gamma}\prec X^{\gamma}$.
\end{lemma}
\begin{proof}
Let $\nu(\Gamma)$ be the collection of subgroups of $\Gamma$ of the form
$\Ker\rho$, where $\rho\colon \Gamma\to\CC^*$ runs over the set of characters
appearing in the action of $\Gamma$ on the fibers of the normal bundle
of the inclusion of $X^{\Gamma}$ in $X$.
Since $\Gamma$ is finite, its representations are rigid, so they are locally
constant on $X^{\Gamma}$. On the other hand, for each $x\in X^{\Gamma}$ the
representation of $\Gamma$ on $T_xX/T_x^{\Gamma}$ splits as the sum of at most
$\dim X$ different irreducible representations. Consequently, $\nu(\Gamma)$
has at most $\dim X |\pi_0(X^{\Gamma})|$ elements.
By Corollary \ref{cor:cohom-no-augmenta}, $|\pi_0(X^{\Gamma})|\leq\sum_jb_j(X;\FF_p)$.
On the other hand, since $\Gamma$ is $K$-stable, we have $|\Gamma'|\leq K^{-1}|\Gamma|$
for each $\Gamma'\in\nu(\Gamma)$, so
$$\left|\bigcup_{\Gamma'\in\nu(\Gamma)}\Gamma'\right|\leq K^{-1}|\Gamma|
|\nu(\Gamma)|\leq K^{-1}|\Gamma|\dim X\left(\sum_jb_j(X;\FF_p)\right)<|\Gamma|.$$
Consequently, there exists at least one element $\gamma\in\Gamma$ not contained in $\bigcup_{\Gamma'\in\nu(\Gamma)}\Gamma'$.
By Lemma \ref{lemma:linearization} we have $X^{\Gamma}\prec X^{\gamma}$.
\end{proof}

\subsection{$K$-stable actions: arbitrary abelian groups}
Let $\Gamma$ be an abelian group acting on a 
manifold $X$.
Recall that for any prime $p$ we denote by $\Gamma_{p}\subseteq\Gamma$ the $p$-part of Gamma.
We say that the action of $\Gamma$ on $X$ is $K$-stable if and only if
for any prime $p$ the restriction of the action to
$\Gamma_{p}$ is $K$-stable (recall that any action of the trivial group is $K$-stable).
As for $p$-groups, when the manifold $X$ and the action are clear from the context,
we will sometimes say that $\Gamma$ is $K$-stable (this will be often the case
when talking about subgroups of a group acting on $X$).

\begin{theorem}
\label{thm:existence-stable-subgroups}
Let $K\geq 1$ be a real number.
There exists a constant $\Lambda(K)$, depending only on $X$,
such that any abelian group $\Gamma$ acting on $X$ in a CT-way
has a $K$-stable subgroup of index at most $\Lambda(K)$.
\end{theorem}
\begin{proof}
By (\ref{eq:CT-implies-CT-p}), the hypothesis that $\Gamma$ acts in a CT way implies that
it acts on a $\FF_p$-CU way for every prime $p$.
Let $P_{\chi}$ be the number defined in Lemma \ref{lemma:exists-Gamma-chi-p}.
Define
$$\Lambda(K):=\left(\prod_{p\leq P_{\chi}}C_{p,\chi}\right)\left(\prod_{p\leq K-1}C_{p,K}\right),$$
where in both products $p$ runs over the set of primes satisfying the inequality.
The theorem follows from combining Lemma \ref{lemma:exists-Gamma-chi-p},
Lemma \ref{lemma:exists-Gamma-stable-p} and Lemma \ref{lemma:mixing-stable}
applied to each of the factors
of $\Gamma\simeq\prod_{p|d}\Gamma_p$, where $d=|\Gamma|$.
\end{proof}

\section{$K$-stable actions on manifolds without odd cohomology}
\label{s:K-stable-actions-no-odd-cohomology}

In this section $X$ denotes a 
manifold without odd cohomology.
Let $p$ be any prime number.
Applying cohomology to the long exact sequence (\ref{eq:ex-sequence-z-p})
and using the fact that $X$ has no odd cohomology we obtain
\begin{equation}
\label{eq:betti-numbers-the-same}
b_j(X;\FF_p)=b_j(X)\quad\text{for any $j$}\qquad\Longrightarrow\qquad
\chi(X)=\sum_jb_j(X)=\sum_jb_j(X;\FF_p).
\end{equation}

\subsection{Results on stable actions of abelian $p$-groups}

\begin{lemma}
\label{lemma:no-odd-cohomology}
Let $p$ be any prime
number. Suppose that a $p$-group $\Gamma$ acts on $X$ in a $\FF_p$-CU way,
and that there is a subgroup $\Gamma'\subseteq \Gamma$ such that
$X^{\Gamma}\prec X^{\Gamma'}$ and $\chi(X^{\Gamma})=\chi(X^{\Gamma'})=\chi(X)$.
Then $X^{\Gamma}=X^{\Gamma'}$.
\end{lemma}
\begin{proof}
If $\Gamma$ is a $p$-group acting on $X$ in a $\FF_p$-CU way then
by Corollary \ref{cor:cohom-no-augmenta} we have $\sum_jb_j(X^{\Gamma};\FF_p)\leq\sum_jb_j(X;\FF_p)$
so, using (\ref{eq:betti-numbers-the-same}),
$$|\chi(X^{\Gamma})|\leq \sum_jb_j(X^{\Gamma};\FF_p)\leq\sum_jb_j(X;\FF_p)=\chi(X).$$
If in addition $\chi(X^{\Gamma})=\chi(X)$, then all inequalities are equalities, so in particular we have
$\sum_jb_j(X^{\Gamma};\FF_p)=\sum_jb_j(X;\FF_p)$. Now, given an inclusion of groups
$\Gamma'\subseteq \Gamma$ satisfying the hypothesis of the lemma, the previous arguments give
$\sum_jb_j(X^{\Gamma};\FF_p)=\sum_jb_j(X^{\Gamma'};\FF_p)$. If this is combined
with $X^{\Gamma}\prec X^{\Gamma'}$ then we deduce $X^{\Gamma}=X^{\Gamma'}$.
\end{proof}

\begin{lemma}
\label{lemma:good-p-subgroup-of-stable}
Let $\Gamma$ be an abelian $p$-group, acting on $X$ in a $K$-stable way. Suppose
that $\Gamma_0\subseteq \Gamma$ is a subgroup satisfying $d:=[\Gamma:\Gamma_0]<K$.
Then the action of $\Gamma_0$ on $X$ is $K/d$-stable and $X^{\Gamma_0}=X^{\Gamma}$.
\end{lemma}
\begin{proof}
Property (1) in the definition of $K$-stability is independent of $K$ and is obviously
inherited by subgroups, so it suffices to prove property (2).
By Lemma \ref{lemma:subgroup-stable-general} we have $X^{\Gamma}\prec X^{\Gamma_0}$.
Since $\Gamma$ is $K$-stable, we have $\chi(X^{\Gamma})=\chi(X^{\Gamma_0})=\chi(X)$.
Hence, Lemma \ref{lemma:no-odd-cohomology} gives $X^{\Gamma_0}=X^{\Gamma}$.
It follows that any character $\rho_0\colon \Gamma_0\to\CC^*$ appearing as a summand in the
action of $\Gamma_0$ in one of the fibers of the normal bundle of the inclusion
$X^{\Gamma_0}\hookrightarrow X$ is the restriction of a character $\rho\colon \Gamma\to\CC^*$
appearing in the action of $\Gamma$ on the normal bundle of $X^{\Gamma}\hookrightarrow X$;
by the hypothesis that $\Gamma$ is $K$-stable, we have $[\Gamma:\Ker\rho]\geq K$.
Consequently:
$$[\Gamma_0:\Gamma_0\cap\Ker\rho]\geq \frac{[\Gamma:\Gamma_0\cap \Ker\rho]}{[\Gamma:\Gamma_0]}
\geq \frac{[\Gamma:\Ker\rho]}{[\Gamma:\Gamma_0]}=K/d,$$
so the lemma is proved.
\end{proof}

\begin{lemma}
\label{lemma:good-p-Gamma-has-gamma}
Let $K_\chi:=\chi(X)\dim X+1$.
If $\Gamma$ is an abelian $p$-group acting on $X$ in a $K_\chi$-stable way,
then there exists some $\gamma\in X$ such that $X^{\Gamma}=X^{\gamma}$.
\end{lemma}
\begin{proof}
This follows from combining Lemma \ref{lemma:existeix-element-generic-p},
 equality (\ref{eq:betti-numbers-the-same}), and Lemma \ref{lemma:no-odd-cohomology}.
\end{proof}

\subsection{Results on stable actions of arbitrary abelian groups}

\begin{theorem}
\label{thm:good-Gamma-has-gamma}
Define $K_{\chi}$ as in Lemma \ref{lemma:good-p-Gamma-has-gamma}.
If an action of an abelian group $\Gamma$ on $X$ is $K_\chi$-stable, then
for any connected component $Y\subset X$ we have $\chi(Y^{\Gamma})=\chi(Y)$.
\end{theorem}
\begin{proof}
There is an isomorphism $\Gamma\simeq \Gamma_{p_1}\times\dots\times \Gamma_{p_k}$,
where $p_1,\dots,p_k$ are the prime divisors of $|\Gamma|$.
Since the action of $\Gamma$ is $K_\chi$-stable so is, by definition,
its restriction to each $\Gamma_{p_i}$, so by
Lemma \ref{lemma:good-p-Gamma-has-gamma} there exists, for each $i$, an element
$\gamma_i\in\Gamma_{p_i}$ such that $X^{\gamma_i}=X^{\Gamma_{p_i}}$.
Let $\gamma=\gamma_1\dots\gamma_k$. Then
$X^{\Gamma}=\bigcap_i X^{\Gamma_i}\subseteq X^{\gamma}.$
By the Chinese remainder theorem and the fact that the elements $\gamma_1,\dots,\gamma_k$
commute,
for each $i$ there exists some $e$ such that $\gamma^e=\gamma_i$.
Hence
$X^{\gamma}\subseteq X^{\gamma^e}=X^{\gamma_i}=X^{\Gamma_{p_i}}.$
Taking the intersection over all $i$ we get
$X^{\gamma}\subset \bigcap_i X^{\Gamma_i}=X^{\Gamma}.$
Combining the two inclusions we have $X^{\gamma}=X^{\Gamma}$.
By definition, the assumption that the action of $\Gamma$ on $X$
is $K_{\chi}$-stable implies that it is CT, so in particular $\Gamma$
preserves the connected components of $X$, and its restriction to each
connected component is CT. So the lemma
follows from applying Lemma \ref{lemma:Euler-fixed-point-set-cyclic}
to each connected component of $X$.
\end{proof}

\begin{lemma}
\label{lemma:good-subgroup-of-stable}
Let $\Gamma$ be an abelian group, acting on $X$ in a $K$-stable way. Suppose
that $\Gamma_0\subseteq \Gamma$ is a subgroup satisfying $d:=[\Gamma:\Gamma_0]<K$.
Then the action of $\Gamma_0$ on $X$ is $K/d$-stable and $X^{\Gamma_0}=X^{\Gamma}$.
\end{lemma}
\begin{proof}
Apply Lemma \ref{lemma:good-p-subgroup-of-stable} to each
factor in $\Gamma\simeq\prod_{p|d}\Gamma_p$, with $d=|\Gamma|$.
\end{proof}

\subsection{$p$-depth of abelian subgroups}
\label{ss:p-depth}
Let $G$ be a group acting on a manifold $X$ in a CT way.
Let $\Gamma\subseteq G$ be an abelian subgroup and let $p$ be a prime.
Define the {\bf $p$-depth} of $\Gamma$ (with respect to $X$ and the action of $G$)
as the biggest $k$ for which there exists a collection of abelian $p$-subgroups $\Gamma(0),\dots,\Gamma(k)$ of $G$, such that
$\Gamma(0)$ is equal to the $p$-part of $\Gamma$ and there are inclusions
$$X^{\Gamma(k)}\subsetneq X^{\Gamma(k-1)}\subsetneq\dots\subsetneq X^{\Gamma(1)}\subsetneq X^{\Gamma(0)}.$$
(Remark that we are not assuming inclusions among the subgroups $\Gamma(j)$, but only
among their fixed point loci.)
We denote the $p$-depth of $\Gamma$ as $\depth_p^{G,X}(\Gamma)$.

\begin{lemma}
\label{lemma:depth-bounded}
There exists a constant $C_{\depth}(X)$, depending only on $X$,
such that $$\depth_p^{G,X}(\Gamma)\leq C_{\depth}(X)$$ for any prime $p$,
any group $G$ acting in a CT way on $X$, and any abelian subgroup $\Gamma\subseteq G$.
\end{lemma}
\begin{proof}
By Corollary \ref{cor:cohom-no-augmenta}, for any abelian $p$-group $\Gamma_p$ acting
in a $\FF_p$-CU way on $X$ the number of connected components of $X^{\Gamma}$ is at most $\chi(X)$.
So by Lemma \ref{lemma:chains-inclusions} the number $C_{\depth}(X):=C_M(\dim X,\chi(X))$, where
$C_M$ is defined in Lemma \ref{lemma:chains-inclusions}, is an upper bound for the $p$-depth
of any abelian group $\Gamma$ satisfying the hypothesis of the lemma.
\end{proof}

\subsection{Intersections of fixed points sets of $K$-stable abelian subgroups}
\label{ss:intersections-fixed-point-sets}
In this subsection we assume that $X$ is connected.
Let $C_{\Jord}$ be the constant in Jordan's Theorem \ref{thm:Jordan-classic}
for finite subgroups of $\GL(\dim X,\RR)$. Since $X$ has no odd cohomology, the constant
$P_{\chi}$ given by
Lemma \ref{lemma:exists-Gamma-chi-p} can be taken to be $P_{\chi}=2\chi(X)$.
Define
$$C_{\chi}=\prod_ {p\leq 2\chi(X)}C_{p,\chi},$$
where $p$ runs over the set of primes satisfying the inequality.

\begin{theorem}
\label{thm:intersection-stabilizes-any-group}
Let $G$ be a group acting on $X$ in a ECT-way.
Let $\Gamma_1,\dots,\Gamma_s$ be $K$-stable abelian
subgroups of $G$, where
$K$ is any number strictly bigger than $C_{\Jord}C_{\chi}$
and $s$ is arbitrary.
If $X^{\Gamma_1}\cap\dots\cap X^{\Gamma_s}\neq\emptyset$
then there exists a $K/(C_{\Jord}C_{\chi})$-stable abelian group
$\Gamma\subseteq G$ such that
$$X^{\Gamma_1}\cap\dots\cap X^{\Gamma_s}=X^{\Gamma}$$
and such that for any subgroup $H\subseteq G$ we have
\begin{equation}
\label{eq:interseccio-segueix-sent-gran}
[\Gamma:H\cap\Gamma]\geq \frac{\max_j[\Gamma_j:H\cap\Gamma_j]}{(C_{\Jord}C_{\chi})^{\max\{2\chi(X),C_{\Jord}\}}}.
\end{equation}
Furthermore, if there is some prime $p$ such that
each $\Gamma_j$ is a $p$-group then the group $\Gamma$
can be chosen to be a $p$-group.
\end{theorem}
\begin{proof}
Let $R\subseteq G$ be the subgroup generated by $\{\Gamma_j\}$.
Since $X$ is connected, statement (2) in Lemma \ref{lemma:linearization} implies
that for any $x\in\bigcap_jX^{\Gamma_j}$ there is an inclusion $R\hookrightarrow\GL(T_xX)$.
Applying Jordan's Theorem \ref{thm:Jordan-classic} we deduce that there
exists an abelian subgroup $R^a\subseteq R$ of index at most $C_{\Jord}$.
Let $R_p^a$ be the $p$-part of $R^a$.
By Lemma \ref{lemma:exists-Gamma-chi-p} there exists a subgroup
$R_p^{a,\chi}\subseteq R_p^a$ of index at most $C_{p,\chi}$
such that
\begin{equation}
\label{eq:prop-def-R-chi}
\text{for any subgroup $Q\subseteq R_p^{a,\chi}$ we have $\chi(X^Q)=\chi(X)$.}
\end{equation}
Let $R^{a,\chi}=\prod_p R^{a,\chi}_p$. Since $C_{p,\chi}=1$ for all
primes bigger than $2\chi(X)$, $R^{a,\chi}$ has index at most $C_{\Jord}C_{\chi}$
in $R$. Hence, $[\Gamma_j:R^{a,\chi}\cap \Gamma_j]\leq C_{\Jord}C_{\chi}$
for any $j$, and similarly for any prime $p$ we have
\begin{equation}
\label{eq:R-prop-de-Gamma}
[\Gamma_{j,p}:(R^{a,\chi}\cap \Gamma_j)_p]\leq C_{\Jord}C_{\chi}.
\end{equation}
By Lemma \ref{lemma:good-p-subgroup-of-stable}, this implies that
$(R^{a,\chi}\cap \Gamma_j)_p$ is $K/(C_{\Jord}C_{\chi})$-stable and that
$$X^{\Gamma_{j,p}}=X^{(R^{a,\chi}\cap \Gamma_j)_p}$$
(here we use the hypothesis $K/(C_{\Jord}C_{\chi})\geq 1$).
Now define $\Gamma(p)$ to be the subgroup of $R^{a,\chi}$ generated by
the groups $\{(R^{a,\chi}\cap \Gamma_j)_p\}_{1\leq j\leq s}$. By (\ref{eq:prop-def-R-chi}),
we can apply Lemma \ref{lemma:mixing-stable} recursively to the subgroup
of $R^{a,\chi}$ generated by $$(R^{a,\chi}\cap \Gamma_1)_p,\dots,(R^{a,\chi}\cap \Gamma_l)_p,$$
with $l$ moving from $2$ to $s$, and deduce that
$\Gamma(p)$ is $K/(C_{\Jord}C_{\chi})$-stable. Since each $\Gamma(p)$ is contained
in the abelian group $R^a$, the product $\Gamma:=\prod_p\Gamma(p)$ is an abelian group
whose $p$-part is $\Gamma(p)$. Hence $\Gamma$ is $K/(C_{\Jord}C_{\chi})$-stable
and we have
$$X^{\Gamma}=\bigcap_p X^{\Gamma(p)}=\bigcap_{p,j}X^{(R^{a,\chi}\cap \Gamma_j)_p}
=\bigcap_{p,j}X^{\Gamma_{j,p}}=\bigcap_j X^{\Gamma_j}.$$

In order to prove the second statement of the theorem, we compute
\begin{align*}
[\Gamma_{j,p}:(\Gamma\cap \Gamma_j)_p] &= [\Gamma_{j,p}:\Gamma(p)\cap\Gamma_{j,p}] \\
&\leq [\Gamma_{j,p}:(R^{a,\chi}\cap\Gamma_j)_p]\qquad\qquad\text{because
$(R^{a,\chi}\cap\Gamma_j)_p\subseteq\Gamma(p)$} \\
&\leq C_{\Jord}C_{\chi} \,\,\,\,\,\qquad\qquad\qquad\qquad\text{by (\ref{eq:R-prop-de-Gamma})}.
\end{align*}
If $p>\max\{2\chi(X),C_{\Jord}\}$ then the above estimate can be substantially improved
as follows. First of all, since $[R:R^a]\leq C_{\Jord}$, we have
$[\Gamma_j:R^a\cap \Gamma_j]\leq C_{\Jord}$, and taking the $p$-part (i.e.,
intersecting with $\Gamma_{j,p}$) we have
$[\Gamma_{j,p}:R^a_p\cap \Gamma_{j,p}]\leq C_{\Jord}$. Since $\Gamma_{j,p}$ is a $p$-group
and $p>C_{\Jord}$, this implies that
$[\Gamma_{j,p}:R^a_p\cap \Gamma_{j,p}]=1$, so $\Gamma_{j,p}\subseteq R^a_p$.
On the other hand, since $p>2\chi(X)$ we have $C_{p,\chi}=1$, so $R^a_p=R^{a,\chi}_p$.
It follows that in this case $[\Gamma_{j,p}:(R^{a,\chi}\cap\Gamma_j)_p]=1$, so
$[\Gamma_{j,p}:(\Gamma\cap \Gamma_j)_p]=1$.

Hence, we may estimate
\begin{align*}
[\Gamma_j:\Gamma\cap\Gamma_j] &= \prod_p [\Gamma_{j,p}:(\Gamma\cap \Gamma_j)_p]
=\prod_{p\leq\max\{2\chi(X),C_{\Jord}\}} [\Gamma_{j,p}:(\Gamma\cap \Gamma_j)_p] \\
&\leq (C_{\Jord}C_{\chi})^{\max\{2\chi(X),C_{\Jord}\}}.
\end{align*}
Now let $H\subseteq G$ be any subgroup. We have, for any $j$:
\begin{align*}
[\Gamma:H\cap\Gamma] &\geq [\Gamma\cap\Gamma_j:H\cap\Gamma\cap\Gamma_j]=
\frac{|\Gamma\cap\Gamma_j|}{|H\cap\Gamma\cap\Gamma_j|}\geq
\frac{|\Gamma\cap\Gamma_j|}{|H\cap\Gamma_j|} \\
&=\frac{1}{[\Gamma_j:\Gamma\cap\Gamma_j]}\frac{|\Gamma_j|}{|H\cap\Gamma_j|}
\geq
\frac{[\Gamma_j:H\cap\Gamma_j]}{(C_{\Jord}C_{\chi})^{\max\{2\chi(X),C_{\Jord}\}}}.
\end{align*}
This proves (\ref{eq:interseccio-segueix-sent-gran}).

Finally, if there is some prime $p$ such that each $\Gamma_j$ is a $p$-group
then $\Gamma=\Gamma(p)$, since for any primer $q\neq p$ the group $\Gamma(q)$
is trivial. Hence $\Gamma$ is a $p$-group.
\end{proof}

\section{Proof of Theorem \ref{thm:main-abelian}}
\label{s:proof-thm:main-abelian}

By Lemma \ref{lemma:Minkowski} it suffices to prove the following.

\begin{theorem}
\label{thm:main-abelian-odd-CT}
Let $X$ be a manifold without odd cohomology,
with connected components $X_1,\dots,X_r$. There exists a number $C_a(X)\in\NN$
such that any abelian group $A'$ acting in an ECT way on $X$ has a subgroup $A$ of
index at most $C_a(X)$ such that $\chi(X_i^{A})=\chi(X_i)$ for every $i$.
Furthermore $A$ can be generated by at most $\sum_i[\dim X_i/2]$ elements.
\end{theorem}
\begin{proof}
Let $C_{\dim X_i}$ (resp. $K_{\chi}$) be the constant defined in Theorem \ref{thm:abelian-generat}
(resp. Theorem \ref{thm:good-Gamma-has-gamma}.)
Let $K:=\max\{1+\prod_i C_{\dim X_i},K_{\chi}\}$.
By Theorem \ref{thm:existence-stable-subgroups} there exists a $K$-stable subgroup
$A''\subseteq A'$
of index at most $\Lambda(K)$.
Since $K\geq K_{\chi}$,
Theorem \ref{thm:good-Gamma-has-gamma} implies that
$\chi(X_i^{A''})=\chi(X_i)\neq 0$ for every $i$, so there exists some point
$x_i\in X_i$ fixed by all elements in $A''$.

Since the action of $A$ on $X$ is CT, it preserves the connected components of $X$.
Let $A_i''\subset\Diff(X_i)$ be the image of $A''$ under the morphism
$A''\to\Diff(X_i)$ induced by the action (the assumption that $A'$ acts
effectively on $X$ does not necessarily imply that the restriction to
each connected components is effective).
By statement (2) in Lemma \ref{lemma:linearization},
we can identify $A_i''$ with a subgroup of
$\Aut(T_{x_i}X_i)\simeq \GL(\dim X_i,\RR)$. Since furthermore
$A_i''$ is abelian, by Theorem \ref{thm:abelian-generat} there exists a subgroup
$A_{i}\subseteq A_i''$ of index at most $C_{\dim X_i}$ which can be generated by
at most $[\dim X_i/2]$ elements.
For each $i$ let $\pi_i:A''\to A_i''$ be the natural projection.
Let $A:=\bigcap_i\pi_i^{-1}(A_{i})$. Since
$A''$ is $K$-stable and $[A'':A]\leq \prod_iC_{\dim X_i}<K$,
Lemma \ref{lemma:good-subgroup-of-stable} implies that $X^{A}=X^{A''}$, 
which gives $\chi(X_i^{A})=\chi(X_i)$ for every $i$.
On the other hand, we have $[A'':A]\leq C_a(X):=\Lambda(K) \prod_iC_{\dim X_i}$.

We now construct a generating subset $\{e_{ij}\}\subseteq A$ with no more
than $\sum[\dim X_i/2]$ elements.
We use additive notation on $A'$.
Let $B_0:=A$ and, for every $1\leq i\leq r$, $B_i:=A\cap \Ker\pi_1\cap\dots\cap\Ker\pi_{i}$.
Since the action of $A'$ on $X$ is effective, we have $\bigcap_{1\leq i\leq r}\Ker\pi_i=\{0\}$,
so $B_{r}=\{0\}$.
Taking into account that if an abelian subgroup can
be generated by at most $d$ elements then the same holds for all of its subgroups,
we may take, for every $i$, elements $e_{i1},\dots,e_{id_i}\in B_{i-1}$,
with $d_i\leq[\dim X_i/2]$,
such that $\pi_i(e_{i1}),\dots,\pi_i(e_{id_i})$ generate
$\pi_i(B_{i-1})\subset A_i$.
We now prove that $\{e_{ij}\}$ generates $A$. Let $a\in A$ be any element.
Define recursively elements $a_0,\dots,a_{r}\in A$, satisfying
$a_i\in B_{i}$ for every $i$,
as follows. Set $a_0:=a$. If $i\geq 1$
and $a_{i-1}\in B_{i-1}$ has been defined, take integers $\lambda_{ij}$ such that
$\pi_i(a_{i-1})=\sum_j\lambda_{ij}\pi_i(e_{ij})$.
Then define $a_i:=a_{i-1}-\sum_j\lambda_{ij}e_{ij}$. It follows that
$a_i\in B_i$. Repeating the procedure until $i=r$ we deduce that $a_{r}\in B_r$, so $a_r=0$.
Consequently, $a=\sum_{i,j}\lambda_{ij}e_{ij}$ and the proof is complete.
\end{proof}

\section{Proof of Theorem \ref{thm:bounded-primes}}
\label{s:proof-thm:bounded-primes}

By Lemma \ref{lemma:Minkowski}, Theorem \ref{thm:bounded-primes} can be reduced
to the following statement.

\begin{theorem}
\label{thm:bounded-primes-CT}
Let $X$ be a connected
manifold without odd cohomology and let $r\in\NN$. There exists an integer
$C(X,r)\geq 1$ such that any finite group $G$ acting in an ECT way on $X$ and satisfying
$|\pi(G)|\leq r$ contains an abelian subgroup $A\subseteq G$ of index at most $C(X,r)$
which can be generated by at most $[\dim X/2]$ elements.
\end{theorem}

The proof of this theorem will occupy this entire section.

Fix once and for all a (compact, connected, possibly with boundary)
manifold $X$ without odd cohomology and a natural number $r$.

\subsection{Review of previous results}
\label{ss:review-previous-results}
Before proving the theorem, let us summarize the notation and
results from the preceding sections which will be used in the course of the proof.
As in Subsection \ref{ss:intersections-fixed-point-sets},
we denote by $C_{\Jord}$ the constant in the statement
of Jordan's Theorem \ref{thm:Jordan-classic} for finite subgroups of $\GL(\dim X,\RR)$.
Given any $\kappa\geq 1$, Theorem \ref{thm:existence-stable-subgroups} states that there exists a constant $\Lambda(\kappa)$
depending only on $X$, such that for any abelian group $\Gamma$ acting
in a ECT way on $X$ there is a $\kappa$-stable abelian subgroup of $\Gamma$ of index at most $\Lambda(\kappa)$.
On the other hand, Theorem \ref{thm:intersection-stabilizes-any-group}
states that for any $\kappa\geq 1$ there exists a constant
$$I(\kappa):=C_{\Jord}C_{\chi}\kappa \geq 1$$
such that for any finite group $G$ acting on $X$ in an ECT way,
and any arbitrary collection $\Gamma_1,\dots,\Gamma_s$
of $I(\kappa)$-stable abelian subgroups of $G$, either $\bigcap X^{\Gamma_i}=\emptyset$
or there exists a $\kappa$-stable abelian
subgroup $\Gamma\subseteq G$ such that $\bigcap X^{\Gamma_i}=X^{\Gamma}$.
Furthermore, $\Gamma$ can be chosen in such a way that
for any subgroup $H\subseteq G$ we have
$[\Gamma:H\cap\Gamma]\geq \alpha^{-1}\max_j[\Gamma_j:H\cap\Gamma_j]$,
where
$$\alpha=(C_{\Jord}C_{\chi})^{\max\{2\chi(X),C_{\Jord}\}}.$$

For any $\kappa'\geq 1$
we say that an abelian subgroup $\Gamma\subseteq G$ is {\bf $\kappa'$-acceptable} if there is
a collection of $\kappa'$-stable subgroups of $\Gamma$, $\Gamma_1,\dots,\Gamma_s$ (where $s$ is
arbitrary), which generate $\Gamma$ and which satisfy $\bigcap_i X^{\Gamma_i}\neq\emptyset$.
(If $\Gamma$ satisfies, additionally, that $\chi(X^{\Gamma_0})=\chi(X)$ for any subgroup
$\Gamma_0\subseteq\Gamma$, then Lemma \ref{lemma:mixing-stable} implies that $\Gamma$ is
$\kappa'$-stable; but note that in general we don't include this condition in the definition
of acceptable subgroup.)
Theorem \ref{thm:intersection-stabilizes-any-group} implies that for any $I(\kappa)$-acceptable
subgroup $\Gamma\subseteq G$ there exists a $\kappa$-stable subgroup $\Gamma'\subseteq G$
such that
$X^{\Gamma}=X^{\Gamma'}$.

Let $K_{\chi}$ be the number defined in Theorem \ref{thm:good-Gamma-has-gamma}.
Recall that if an abelian subgroup $H\subseteq G$ is $K_{\chi}$-stable then $\chi(X^H)=\chi(X)$.
Define
$$K:=I(K_{\chi}).$$
By the previous observations, if $\Gamma\subseteq G$ is $K$-acceptable then $\chi(X^{\Gamma})=\chi(X)$.

\subsection{Induction scheme}
\label{ss:induction-scheme}

Let $\delta$ be a nonnegative integer. Consider the following statement
(recall that the notation $\depth_p^{G,X}$ has been
defined in Subsection \ref{ss:p-depth}):

\begin{quote}
\noindent{\bf T($\delta$).}
There exists a constant $C_T({\delta})$
with the following significance. Let $G$ be a  group
acting on $X$ in an ECT way. Suppose that $|\pi(G)|\leq r$.
Let $\Gamma\subseteq G$ be an abelian $K$-acceptable subgroup satisfying
$$\sum_{p\in\pi(G)}\depth_p^{G,X}(\Gamma)\leq\delta.$$
Let $C_G(\Gamma)$ denote the centraliser of $\Gamma$ in $G$.
Then $C_G(\Gamma)$ has an abelian subgroup $A$ of index at most $C_T({\delta})$
satisfying $X^{A}\cap X^{\Gamma}\neq\emptyset$; furthermore, $A$
can be chosen to be isomorphic to a subgroup of $\GL(\dim X,\RR)$.
\end{quote}

By Lemma \ref{lemma:depth-bounded},
statement T($rC_{\depth}(X)$), taking $\Gamma=\{1\}$, implies
Theorem \ref{thm:bounded-primes}, except for the
bound on the number of generators.
So our next aim is to prove statements $\{\text{T}(\delta)\}$ using ascending induction on
$\delta$.

\subsection{Overview of the proof}
Before entering into the details, it might be useful to give an overview of the
main arguments.
The following notation will be used once a choice of an abelian $K$-acceptable
subgroup $\Gamma\subseteq G$ has been made:
\begin{equation}
\label{eq:def-B}
Y=X^{\Gamma},\qquad
Z=C_G(\Gamma), \qquad
Z_Y=\bigcap_{y\in Y}Z_y,\qquad
\Pi\colon Z\to R:=Z/Z_Y,
\end{equation}
where $\Pi$ is the projection.
Note that, since $\Gamma$ is $K$-acceptable, we have $\chi(Y)=\chi(X)$. Furthermore, $Y$ is $Z$-invariant.

The idea to prove $\text{T}(0)$ is to bound separately the size of every $p$-Sylow subgroup
of $R$, assuming that $\depth^{G,X}_p(\Gamma)=0$. A fixed point argument
(Lemma \ref{lemma:B} below) proves that if a $p$-Sylow subgroup $R_p\subset R$ is
big enough (independently of $G$) then we can
find some $p$-group $\Gamma'\subseteq Z$ such that $X^{\Gamma'}\subsetneq X^{\Gamma_p}$,
contradicting $\depth^{G,X}_p(\Gamma)=0$. Since $|R|$ has at most $r$ different
prime divisors, we get a uniform (i.e., independent of $G$)
upper bound on $|R|=[Z:Z_Y]$. {\it A fortiori}, we have
an upper bound on $[Z:Z_y]$ for every $y\in Y$. By
Lemma \ref{lemma:linearization} there is a monomorphism
$Z_y\hookrightarrow\GL(T_yX)$, and by Jordan's Theorem \ref{thm:Jordan-classic}
$Z_y$ has an abelian subgroup of bounded index, and this completes the proof of
$\text{T}(0)$.

The induction step is more involved.
Assume that $\delta>0$ and that $\text{T}(\delta-1)$ is true.
To prove $\text{T}(\delta)$ we introduce some big constants $K_3$ and $K_5$
(depending on $X$ and $C_T(\delta-1)$) and consider the collection
$\fF$ of nonemtpy subsets of $Y$ of the form $Y^{\Theta}$, where
$\Theta$ is a $K_3$-stable abelian subgroup of $Z$ satisfying $|\Pi(\Theta)|\geq K_5$.
The choice of $K_3$ and $K_5$ guarantees that the relation $\approx$ in $\fF$,
which identifies $F,F'\in\fF$ whenever $F\cap F'\neq\emptyset$, is an equivalence relation
(here the induction hypothesis is used, see Lemma \ref{lemma:intersection-transitive})
and that, denoting by brackets the $\approx$-equivalence classes,
\begin{equation}
\label{eq:ingredient-1}
[F_1]=\dots=[F_s]\quad\Longrightarrow\quad
\chi(F_1\cap\dots\cap F_s)=\chi(X)\qquad\text{for any $F_1,\dots,F_s\in\fF$}
\end{equation}
(this is Lemma \ref{lemma:big-intersection}).
Define $\hH:=\fF/\approx$. The group $R$ acts naturally on $\hH$, and
using Lemma \ref{lemma:B} and Sylow's theorem we prove (Lemma \ref{lemma:l-at-most-r}) that
\begin{equation}
\label{eq:ingredient-2}
|\hH/R|\leq r.
\end{equation}
Next we prove (Lemma \ref{lemma:diferencia-divisible}) that
\begin{equation}
\label{eq:ingredient-3}
\chi(X)-\chi\left(\bigcup_{F\in\fF}F\right)\quad
\text{is divisible by $|R|/e$, where $e$ is uniformly bounded above.}
\end{equation}
Combining (\ref{eq:ingredient-1}), (\ref{eq:ingredient-2}) and (\ref{eq:ingredient-3})
together with an elementary arithmetic argument (Lemma \ref{lemma:equacio-diofantina})
we deduce that there exists some $h_1\in\hH$ whose stabilizer $R':=R_{h_1}\subseteq R$ satisfies
$[R:R']\leq C$ for some uniform constant $C$. Let
$$Y_1=\bigcap_{F\in\fF,\,[F]=h_1}F.$$
Then $Y_1$ is preserved by the action of $R$ on $Y$.
Applying an argument similar to the proof of $\text{T}(0)$ we deduce that the
subgroup $R''\subseteq R'$ consisting of elements which act trivially on $Y_1$
satisfies $[R':R'']\leq C'$ for some uniform constant $C'$ (Lemma \ref{lemma:bound-R''}).
The proof of $\text{T}(\delta)$ is completed following the same
ideas as in the final step of the proof of $\text{T}(0)$.

After this rough description, we proceed to give the details.
Next subsection contains two lemmas which will be used repeatedly,
Subsection \ref{ss:T(0)} contains the proof of $\text{T}(0)$,
Subsection \ref{ss:induction-step} proves the induction step,
and Subsection \ref{ss:end-of-proof}
completes the proof of Theorem \ref{thm:bounded-primes-CT}.

\subsection{Auxiliary lemmas}

\begin{lemma}
\label{lemma:A}
Let a group $G$ act on $X$. For any real number $\kappa\geq 1$
and any $x\in X$ the stabilizer $G_x\subseteq G$ has an abelian
$\kappa$-stable subgroup of index at most $C_{\Jord}\Lambda(\kappa)$.
\end{lemma}
\begin{proof}
This follows from Lemma \ref{lemma:linearization}, Jordan's Theorem \ref{thm:Jordan-classic},
and Theorem \ref{thm:existence-stable-subgroups}.
\end{proof}

\begin{lemma}
\label{lemma:B}
Let $p$ be any prime, and assume that a $p$-group $Z_p$ acts on $X$
in an ECT way. Let $Y\subset X$ be
a $Z_p$-invariant submanifold satisfying $\chi(Y)=\chi(X)$.
For any real $\kappa\geq 1$
there exists an abelian $\kappa$-stable subgroup $Q$ of $Z_p$
satisfying $[Z_p:Q]\leq \chi(X)C_{\Jord}\Lambda(\kappa)$ and $Y^Q\neq\emptyset$.
\end{lemma}
\begin{proof}
Applying Lemma \ref{lemma:one-big-stabiliser} to the action of $Z_p$ on $Y$,
we deduce the existence of some $y\in Y$ such that $[Z_p:(Z_p)_y]\leq \chi(X)$.
Applying Lemma \ref{lemma:A} to the action of $(Z_p)_y$ on $X$ we deduce the
existence of an abelian $\kappa$-stable subgroup $Q$ of $(Z_p)_y$ of index
$[(Z_p)_y:Q]\leq C_{\Jord}\Lambda(\kappa)$. It follows that
$[Z_p:Q]\leq \chi(X)C_{\Jord}\Lambda(\kappa)$. Since $Q\subseteq (Z_p)_y$,
we have $y\in Y^Q$, so $Y^Q\neq\emptyset$.
\end{proof}

\subsection{Proof of T($0$)}
\label{ss:T(0)}
Let $G$ be a finite group acting on $X$ in an ECT way such that $|\pi(G)|\leq r$.
Let $\Gamma$ be a $K$-acceptable abelian subgroup of $G$ satisfying
$$\sum_{p\in\pi(G)}\depth_p^{G,X}(\Gamma)=0,$$
which is equivalent to $\depth_p^{G,X}(\Gamma)=0$ for each
$p\in\pi(G)$.
Let $Y$, $Z$, $Z_Y$, and $\Pi\colon Z\to R$ be defined as in (\ref{eq:def-B}).
We claim that the $p$-Sylow subgroups
of $R$ have at most $S:=\chi(X)C_{\Jord}\Lambda(K)$ elements.
Suppose the contrary, and let
$R_p\subseteq R$ be a $p$-Sylow subgroup satisfying $|R_p|>S$.
Let $Z_p\subseteq Z$ be a $p$-subgroup satisfying $\Pi(Z_p)=R_p$.
We have $\chi(Y)=\chi(X)$, so by Lemma \ref{lemma:B} there exists an
abelian $K$-stable subgroup $Q$ of $Z_p$ satisfying
$[Z_p:Q]\leq S$.
Since by assumption $|R_p|>S$,
$\Pi(Q)$ is a nontrivial subgroup of $R_p$, which means that $Q$
is not contained in $Z_Y$.
Hence,
$X^{\Gamma}\not\subseteq X^Q$, so
$X^{\Gamma_p}\not\subseteq X^Q$
(because $X^{\Gamma}\subseteq X^{\Gamma_p}$).
Consequently $X^{Q}\cap X^{\Gamma_p}$ is a proper subset of $X^{\Gamma_p}$. By our choice of $K$
and last statement of Theorem \ref{thm:intersection-stabilizes-any-group},
there exists some abelian $p$-group $\Delta\subseteq G$ such that
$X^{Q}\cap X^{\Gamma_p}=X^{\Delta}$. This contradicts our assumption that $\depth_p^{G,X}(\Gamma)=0$,
so the claim is proved.

Repeating this argument for all Sylow subgroups of $R$, and using the fact that
$\pi(R)\subseteq\pi(G)$, we deduce that
$|R|=[Z:Z_Y]\leq S^r$.
Let $y\in Y$.
Since $y$ is stabilised by $Z_Y$,
by (2) in Lemma \ref{lemma:linearization}
we have an inclusion
$Z_Y\hookrightarrow\GL(T_yX)$
so Jordan's Theorem gives an abelian
subgroup $A_0\subseteq Z_Y$
of index at most $C_{\Jord}$.
In particular, $A_0$ is isomorphic to a subgroup of $\GL(\dim X,\RR)$,
and $[C_G(\Gamma):A_0]=[Z:A_0]\leq S^rC_{\Jord}$. Finally, $A_0$ stabilises $y\in X^{\Gamma}$,
so $X^{A_0}\cap X^{\Gamma}\neq\emptyset$, and the proof of statement T($0$) is complete.

\subsection{Induction step}
\label{ss:induction-step}
Let $\delta\geq 1$, and assume that T($\delta-1$) is true.
Let $G$ be a finite group acting on $X$ in an ECT way and satisfying $|\pi(G)|\leq r$.
Let $\Gamma\subseteq G$ be a $K$-acceptable abelian subgroup satisfying
$\sum_{p\in\pi(G)}\depth_p^{G,X}(\Gamma)\leq\delta$.

Let $Y$, $Z$, $Z_Y$, and $\Pi\colon Z\to R$ be defined as in (\ref{eq:def-B}).

Given two subgroups $G_1,G_2\subseteq G$, we denote by $\la G_1,G_2\ra$
the subgroup of $G$ generated by the elements of $G_1$ and $G_2$.

We are going to use the induction hypothesis in the following lemma.

\begin{lemma}
\label{lemma:induccio-0}
Suppose that two abelian subgroups $A_1,A_2$ of $Z$ satisfy
$|\Pi(A_1\cap A_2)|>\Lambda(K)$ and that $Y^{A_1}\neq\emptyset$.
Then there is an abelian subgroup $A\subseteq \la A_1,A_2\ra$
of index at most $C_T(\delta-1)$ such that $Y^A\neq\emptyset$.
\end{lemma}
\begin{proof}
Let $\Gamma'\subseteq A_1\cap A_2$ be a $K$-stable abelian
subgroup of index at most $\Lambda(K)$. Then
$|\Pi(A_1\cap A_2)|>\Lambda(K)$ implies that
$\Gamma'\not\subseteq Z_Y$.
Since $\Gamma'\subseteq Z=C_G(\Gamma)$, $\Gamma$ and $\Gamma'$
generate an abelian group, say $\Gamma^0\subseteq Z$;
since $X^{A_1}\subseteq X^{\Gamma'}$ (because
$\Gamma'\subseteq A_1$) and by assumption $X^{A_1}\cap X^{\Gamma}\neq\emptyset$, we have
$X^{\Gamma'}\cap X^{\Gamma}\neq\emptyset$; combining this with the fact that $\Gamma'$ is $K$-stable and $\Gamma$ is
$K$-acceptable, we deduce that $\Gamma^0$ is $K$-acceptable. We claim that
\begin{equation}
\label{eq:depth-has-decreased}
\sum_{p\in\pi(G)}\depth_p^{G,X}(\Gamma^0)<\delta.
\end{equation}
Since $\Gamma\subseteq \Gamma^0$, we have $\Gamma_p\subseteq\Gamma^0_p$,
and consequently $X^{\Gamma^0_p}\subseteq X^{\Gamma_p}$, for each $p\in\pi(G)$.
Hence,  $\depth_p^{G,X}(\Gamma^0)\leq \depth_p^{G,X}(\Gamma)$ for each $p$.
There exists at least one $p\in\pi(G)$ for which the inclusion
$X^{\Gamma^0_p}\subseteq X^{\Gamma_p}$ is strict (which implies that
$\depth_p^{G,X}(\Gamma^0)<\depth_p^{G,X}(\Gamma)$). Indeed, otherwise we would have
$X^{\Gamma^0_p}=X^{\Gamma_p}$ for each $p$, which would imply
$X^{\Gamma^0}=\bigcap_{p\in\pi(G)}X^{\Gamma^0_p}=\bigcap_{p\in\pi(G)}X^{\Gamma_p}=X^{\Gamma}$.
This would mean that $\Gamma^0\subseteq Z_Y$, which contradicts
$\Gamma'\not\subseteq Z_Y$,
so (\ref{eq:depth-has-decreased}) holds true.

Since $A_1$ and $A_2$ are abelian and $\Gamma'\subseteq A_1\cap A_2$, we have
$\la A_1,A_2\ra\subseteq C_G(\Gamma')$. Hence,
$\la A_1,A_2\ra\subseteq C_G(\Gamma)\cap C_G(\Gamma')=C_G(\Gamma^0)$.
By the induction hypothesis, $C_G(\Gamma^0)$ contains an abelian subgroup $B$
of index at most $C_T(\delta-1)$ such that $X^{B}\cap X^{\Gamma^0}\neq\emptyset$.
In particular, $A:=\la A_1,A_2\ra \cap B$
is abelian and has index at most $C_T(\delta-1)$ in $\la A_1,A_2\ra$.
Furthermore, since $A\subseteq B$ and
$\Gamma\subseteq \Gamma^0$, we have
$\emptyset\neq X^{B}\cap X^{\Gamma^0}\subseteq X^A\cap X^{\Gamma}=Y^A$, so $Y^A\neq\emptyset$.
\end{proof}

\begin{lemma}
\label{lemma:intersection-transitive}
Let $\Theta_1,\Theta_2,\Theta_3$ be $(C_T(\delta-1)^3+1)$-stable abelian
subgroups of $Z$ satisfying
$Y^{\Theta_j}\neq\emptyset$ and
$$|\Pi(\Theta_j)|>C_{\Jord}\Lambda(K)^2C_T(\delta-1)^4$$
for each $j$.
Suppose that $Y^{\Theta_1}\cap Y^{\Theta_2}\neq\emptyset$
and $Y^{\Theta_2}\cap Y^{\Theta_3}\neq\emptyset$. Then
$Y^{\Theta_1}\cap Y^{\Theta_3}\neq\emptyset$.
\end{lemma}
\begin{proof}
In the computations below we will implicitly use the
fact that all numbers $C_{\Jord}$, $\Lambda(K)$ and $C_T(\delta-1)$ are not smaller than $1$.
Let $x\in Y^{\Theta_1}\cap Y^{\Theta_2}$ be any point.
We have inclusions $\Theta_1,\Theta_2\subseteq Z_x$ and
$\Pi(\Theta_1),\Pi(\Theta_2)\subseteq R_x=\Pi(Z_x)$. In particular, we have
$$|R_x|\geq C_{\Jord}\Lambda(K)^2C_T(\delta-1)^4.$$
Let $Z_x'\subseteq Z_x$ be a $K$-stable abelian subgroup of
index at most $C_{\Jord}\Lambda(K)$.
In the following formulas $j$ denotes either $1$ or $2$.
We have:
$$|\Pi(\Theta_j\cap Z_x')|=\frac{|\Pi(\Theta_j)|}
{[\Pi(\Theta_j):\Pi(\Theta_j\cap Z_x')]}
\geq \frac{|\Pi(\Theta_j)|}{[\Theta_j:\Theta_j\cap Z_x']}
\geq \frac{|\Pi(\Theta_j)|}{C_{\Jord}\Lambda(K)}>\Lambda(K).$$
So by Lemma \ref{lemma:induccio-0} there exists an abelian subgroup
$\Xi_j\subseteq \la\Theta_j,Z_x'\ra$ of index at most $C_T(\delta-1)$
satisfying $Y^{\Xi_j}\neq\emptyset$.
We have
$[Z_x':Z_x'\cap \Xi_j]\leq C_T(\delta-1)$, so
$$|Z_x'\cap \Xi_1\cap \Xi_2|\geq \frac{|Z_x'|}{C_T(\delta-1)^2}.$$
Combining this estimate with $|Z_x|/|Z_x'|=[Z_x:Z_x']\leq C_{\Jord}\Lambda(K)$
and $|Z_Y|=|Z_x|/|R_x|$ (which follows from the fact that $R_x=Z_x/Z_Y$) we get:
\begin{align*}
|\Pi(\Xi_1\cap\Xi_2)| &\geq |\Pi(Z_x'\cap \Xi_1\cap\Xi_2)|\geq
\frac{|Z_x'\cap \Xi_1\cap\Xi_2|}{|Z_Y|}\geq
\frac{|Z_x'|}{C_T(\delta-1)^2}\frac{|R_x|}{|Z_x|} \\
&\geq
\frac{|R_x|}{C_T(\delta-1)^2 C_{\Jord}\Lambda(K)}>\Lambda(K)
\end{align*}
Using again Lemma \ref{lemma:induccio-0} we deduce that there exists
an abelian subgroup $\Psi_{12}\subseteq \la\Xi_1,\Xi_2\ra$ of index at most $C_T(\delta-1)$
and satisfying $Y^{\Psi_{12}}\neq\emptyset$. Tracing the definitions we estimate
$$[\Theta_j:\Theta_j\cap\Psi_{12}]\leq C_T(\delta-1)^2.$$

Repeating the previous arguments with $j=2,3$ instead of $1,2$ and replacing
$x$ by any point in the intersection $Y^{\Theta_2}\cap Y^{\Theta_3}$ we
construct another abelian subgroup $\Psi_{23}\subseteq Z$ satisfying
$$[\Theta_j:\Theta_j\cap\Psi_{23}]\leq C_T(\delta-1)^2$$
for $j=2,3$.

Since both $\Theta_2\cap\Psi_{12}$ and $\Theta_2\cap\Psi_{23}$
have index at most $C_T(\delta-1)^2$ in $\Theta_2$, we estimate
$$|\Psi_{12}\cap\Psi_{23}|\geq |\Psi_{12}\cap\Psi_{23}\cap\Theta_2|\geq
\frac{|\Theta_2|}{C_T(\delta-1)^4}.$$
Hence, using the
equality $|\Theta_2\cap Z_Y|=|\Theta_2|/|\Pi(\Theta_2)|$, we have:
\begin{align*}
|\Pi(\Psi_{12}\cap\Psi_{23})| &\geq |\Pi(\Psi_{12}\cap\Psi_{23}\cap\Theta_2)|\geq
\frac{|\Psi_{12}\cap\Psi_{23}\cap\Theta_2|}{|\Theta_2\cap Z_Y|} \\
&\geq \frac{|\Theta_2|}{C_T(\delta-1)^4}\frac{|\Pi(\Theta_2)|}{|\Theta_2|}=
\frac{|\Pi(\Theta_2)|}{C_T(\delta-1)^4}>\Lambda(K).
\end{align*}
Since $Y^{\Psi_{12}}\neq\emptyset$, we can apply Lemma \ref{lemma:induccio-0} once
again to conclude the existence of an abelian subgroup
$\Phi\subseteq \la\Psi_{12},\Psi_{23}\ra$ of index at most $C_T(\delta-1)$ such that
$Y^{\Phi}\neq\emptyset$. Then the bounds
$[\Theta_1:\Theta_1\cap\Psi_{12}]\leq C_T(\delta-1)^2$
and $[\Theta_3:\Theta_3\cap\Psi_{23}]\leq C_T(\delta-1)^2$ imply that
$$[\Theta_1:\Theta_1\cap\Phi]\leq C_T(\delta-1)^3,\qquad\qquad
[\Theta_3:\Theta_3\cap\Phi]\leq C_T(\delta-1)^3.$$
Since $\Theta_1$ and $\Theta_3$ are $(C_T(\delta-1)^3+1)$-stable, Lemma \ref{lemma:good-subgroup-of-stable}
implies that
$X^{\Theta_i}=X^{\Theta_i\cap\Phi}$ for $i=1,3$. Finally, since $\Phi$ contains both $\Theta_1\cap\Phi$
and $\Theta_3\cap\Phi$, we conclude
$$\emptyset\neq Y^{\Phi}=X^{\Phi}\cap X^{\Gamma}\subseteq X^{\Theta_1\cap\Phi}\cap X^{\Theta_3\cap\Phi}\cap X^{\Gamma}
=X^{\Theta_1}\cap X^{\Theta_2}\cap X^{\Gamma}=Y^{\Theta_1}\cap Y^{\Theta_3},$$
so the proof of the lemma is complete.
\end{proof}

Define the following constants:
\begin{align*}
K_2 &:= \max\{(C_T(\delta-1)^3+1),K_{\chi}\}, \\
K_3 &:= \max\{I(K_2),\,\chi(X)C_{\Jord}\Lambda(K_2)K_2\}, \\
K_4 &:= C_{\Jord}\Lambda(K)^2C_T(\delta-1)^4+1, \\
K_5 &:= \max\{\alpha K_4,\,(\chi(X)C_{\Jord}^2\Lambda(K)^2\Lambda(K_2)C_T(\delta-1)^4)^r+1\}.
\end{align*}
Let $\gG$ be the collection of $K_3$-stable abelian subgroups $\Theta\subseteq Z$
satisfying
$$|\Pi(\Theta)|\geq K_5,\qquad\qquad Y^{\Theta}\neq\emptyset.$$
Define a relation $\sim$ between the elements of $\gG$ by setting, for every
$\Theta,\Theta'\in\gG$,
$$\Theta\sim\Theta'\qquad\Longleftrightarrow\qquad Y^{\Theta}\cap Y^{\Theta'}\neq\emptyset.$$
By Lemma \ref{lemma:intersection-transitive}, this defines an equivalence relation on $\gG$.

\begin{lemma}
\label{lemma:big-intersection}
Suppose that $\Theta_1,\dots,\Theta_s\in\gG$ belong to the same $\sim$-equivalence class.
Then $Y^{\Theta_1}\cap\dots\cap Y^{\Theta_s}\neq\emptyset$ and
$\chi(Y^{\Theta_1}\cap\dots\cap Y^{\Theta_s})=\chi(X).$
\end{lemma}
\begin{proof}
We use induction on $s$. The cases $s=1,2$ are trivial, so let us assume that $s\geq 3$
and that the statement has been proved for smaller values of $s$. In particular we may
assume that $Y^{\Theta_1}\cap\dots\cap Y^{\Theta_{s-1}}\neq\emptyset$. This implies,
by Theorem \ref{thm:intersection-stabilizes-any-group}, that there exists
a $K_2$-stable abelian subgroup $\Theta\subseteq Z$ such that $|\Pi(\Theta)|\geq K_4$ and
$X^{\Theta_1}\cap\dots\cap X^{\Theta_{s-1}}=X^{\Theta}$, so in particular
$Y^{\Theta_1}\cap\dots\cap Y^{\Theta_{s-1}}=Y^{\Theta}$.
Similarly, since $\Theta_{s-1}\sim\Theta_s$, there exists
a $K_2$-stable abelian subgroup $\Theta'\subseteq Z$
satisfying $|\Pi(\Theta')|\geq K_4$ and $Y^{\Theta_{s-1}}\cap Y^{\Theta_s}=Y^{\Theta'}$. Since we clearly have
$Y^{\Theta}\cap Y^{\Theta_{s-1}}\neq\emptyset$ and
$Y^{\Theta'}\cap Y^{\Theta_{s-1}}\neq\emptyset$, we may apply
Lemma \ref{lemma:intersection-transitive} to $\Theta,\Theta_{s-1},\Theta'$ and deduce
that
$Y^{\Theta}\cap Y^{\Theta'}\neq\emptyset$, which is equivalent to
$Y^{\Theta_1}\cap\dots\cap Y^{\Theta_s}\neq\emptyset$.
To prove the last statement, note that
$Y^{\Theta_1}\cap\dots\cap Y^{\Theta_s}=X^{\Theta_1}\cap\dots\cap X^{\Theta_s}\cap X^{\Gamma}$.
Since $\Gamma$ is $K$-acceptable and each $\Theta_j$ is also $I(K_{\chi})$-stable
(indeed, $K_2\geq K_{\chi}$ implies $K_3=I(K_2)\geq I(K_{\chi})$), there exists a $K_{\chi}$-stable abelian subgroup
$\Delta\subseteq Z$
satisfying $X^{\Delta}=X^{\Theta_1}\cap\dots\cap X^{\Theta_s}\cap X^{\Gamma}$.
This implies, by Theorem \ref{thm:good-Gamma-has-gamma}, that
$\chi(X^{\Theta_1}\cap\dots\cap X^{\Theta_s}\cap X^{\Gamma})=\chi(X^{\Delta})=\chi(X)$.
\end{proof}

Now let $\fF:=\{Y^{\Theta}\mid \Theta\in\gG\}$.
Define a relation
$\approx$ on $\fF$ by declaring that two elements $F,F'\in\fF$
are related,
$F\approx F'$, if and only if $F\cap F'\neq\emptyset$.
By Lemma \ref{lemma:intersection-transitive},
$\approx$ is an equivalence relation. Furthermore, if $F_1,\dots,F_s$ belong
to the same $\approx$-equivalence class, then $\chi(F_1\cap\dots\cap F_s)=\chi(X)$.
The action of $R$ on $Y$ induces an action of $R$ on the set $\fF$,
since for every $\Theta\in\gG$ and any
$\rho\in R$ we have $\rho\cdot Y^{\Theta}=
\rho\cdot Y^{\Pi(\Theta)}=Y^{\rho\Pi(\Theta)\rho^{-1}}=
Y^{\eta\Theta\eta^{-1}}$, where $\eta\in \Pi^{-1}(\rho)$ is any lift of $\rho$.
This action is obviously compatible with the relation $\approx$.
Hence, if we define $\hH:=\fF/\approx$, then $\hH$ inherits an action of $R$.
Let $h_1,\dots,h_l\in\hH$ be elements such that $\hH=\bigsqcup_j Rh_j$. Let also
$$d:=|R|,\qquad\qquad e_j:=|R_{h_j}|,$$
where $R_{h_j}\subseteq R$ is the stabilizer of $h_j$. Define
$$Y^{\gG}:=\bigcup_{F\in\fF}F.$$
It follows from all the previous considerations that
\begin{equation}
\label{eq:euler-Y-gG}
\chi(Y^{\gG})=\chi(X)\cdot|\hH|
=\chi(X)\left(\sum_{j=1}^l \frac{d}{e_j}\right).
\end{equation}

\begin{lemma}
\label{lemma:l-at-most-r}
The number $l$ of $R$-orbits on $\hH$ is at most $r$.
\end{lemma}
\begin{proof}
Fix for each $p\in\pi(G)$ a Sylow subgroup $R_p$ of $R$ and a $p$-Sylow subgroup
$Z_p\subseteq Z$ such that $\Pi(Z_p)=R_p$. By Lemma \ref{lemma:B} there exists
an abelian $K_2$-stable subgroup $Q_p$ of $Z_p$ satisfying
$[Z_p:Q_p]\leq S_{\delta}:=\chi(X)C_{\Jord}\Lambda(K_2)$ and $Y^{Q_p}\neq\emptyset$.

Now let $\Theta\in\gG$ be any element. Since $|\Pi(\Theta)|>(\chi(X)C_{\Jord}^2\Lambda(K)^2\Lambda(K_2)C_T(\delta-1)^4)^r$
and $|\pi(G)|\leq r$, there exists a prime $p\in\pi(G)$ such that the $p$-Sylow subgroup
of $\Pi(\Theta)$
has $>\chi(X)C_{\Jord}^2\Lambda(K)^2\Lambda(K_2)C_T(\delta-1)^4$ elements (since $\Theta$ is
abelian, $\Pi(\Theta)$ has a unique $p$-Sylow subgroup, which is equal to $\Pi(\Theta_p)$).
Let us fix such a $p$.
By Sylow's theorem, there exists some $\eta\in Z$ such that
$\Theta_p\subseteq \eta Z_p\eta^{-1}$. Hence we have
$$|\Pi(\Theta_p\cap \eta Q_p\eta^{-1})|=\frac{|\Pi(\Theta_p)|}
{[\Pi(\Theta_p\cap\eta Z_p\eta^{-1}):\Pi(\Theta_p\cap\eta Q_p\eta^{-1})]}.$$
We have
\begin{align*}
[\Pi(\Theta_p\cap\eta Z_p\eta^{-1}):\Pi(\Theta_p\cap\eta Q_p\eta^{-1})]
&\leq [\Theta_p\cap\eta Z_p\eta^{-1}:\Theta_p\cap\eta Q_p\eta^{-1}] \\
&\leq [\eta Z_p\eta^{-1}:\eta Q_p\eta^{-1}]
=[Z_p:Q_p]\leq S_{\delta},
\end{align*}
which combined with the previous estimates gives
\begin{equation}
\label{eq:interseccio-Theta-Q-gran}
|\Pi(\Theta_p\cap \eta Q_p\eta^{-1})|>\frac{\chi(X)C_{\Jord}^2\Lambda(K)^2\Lambda(K_2)C_T(\delta-1)^4}
{\chi(X)C_{\Jord}\Lambda(K_2)}=C_{\Jord}\Lambda(K)^2C_T(\delta-1)^4.
\end{equation}
Since $\Theta_p$ is $K_3$-stable and $[\Theta_p:\Theta_p\cap \eta Q_p\eta^{-1}]\leq S_{\delta}$,
we deduce (using Lemma \ref{lemma:good-p-subgroup-of-stable}) that
\begin{equation}
\label{eq:interseccio-Theta-Q-estable}
\Theta_p\cap \eta Q_p\eta^{-1}\qquad\text{is $K_2$-stable}
\end{equation}
because $K_3/[\Theta_p:\Theta_p\cap \eta Q_p\eta^{-1}]\geq K_2$.
We next want to apply Lemma \ref{lemma:intersection-transitive} to
$$\Theta_1 := \Theta, \qquad
\Theta_2 := \Theta_p\cap \eta Q_p\eta^{-1}, \qquad
\Theta_3 := \eta Q_p\eta^{-1}.$$
The group $\Theta_1$ satisfies the hypothesis of Lemma \ref{lemma:intersection-transitive}
by the definition of $\gG$. That $\Theta_2$ satisfies the hypothesis is the content
of statements (\ref{eq:interseccio-Theta-Q-gran}) and (\ref{eq:interseccio-Theta-Q-estable}).
Finally, $\Theta_3$ is $K_2$-stable and $|\Pi(\Theta_3)|\geq |\Pi(\Theta_p\cap \eta Q_p\eta^{-1})|$, so
$|\Pi(\Theta_3)|$ is big enough in view of (\ref{eq:interseccio-Theta-Q-gran}).
We have inclusions $\Theta_2\subseteq\Theta_1$ and $\Theta_2\subseteq\Theta_3$,
which imply $Y^{\Theta_1}\subseteq Y^{\Theta_2}$ and $Y^{\Theta_3}\subseteq Y^{\Theta_2}$.
Since both $Y^{\Theta_1}$ and $Y^{\Theta_3}=\eta Y^{Q_p}$ are nonempty,
we have in particular
$Y^{\Theta_1}\cap Y^{\Theta_2}\neq\emptyset\neq Y^{\Theta_3}\cap Y^{\Theta_2}$.
Now, Lemma \ref{lemma:intersection-transitive} implies $Y^{\Theta_1}\cap Y^{\Theta_3}\neq\emptyset$,
or equivalently $\eta^{-1} Y^{\Theta_1}\cap Y^{Q_p}\neq\emptyset$.

By the previous arguments, we can choose, for any $\Theta\in\gG$, some prime $p_\Theta\in\pi(G)$
and some $\eta_\Theta\in Z$ such that $Q_{p_{\Theta}}$ satisfies the hypothesis of Lemma \ref{lemma:intersection-transitive}
and
$$\eta_{\Theta}^{-1}Y^{\Theta}\cap Y^{Q_{p_{\Theta}}}\neq\emptyset$$
(neither $p_{\Theta}$ nor $\eta_{\Theta}$ are necessarily unique; we just make a choice).
Now, if $\Theta,\Theta'$ satisfy $p_{\Theta}=p_{\Theta'}$ then we can apply
Lemma \ref{lemma:intersection-transitive} to $\Theta_1:=\eta_{\Theta}^{-1}\Theta \eta_{\Theta}$, $\Theta_2:=Q_{p_{\Theta}}=Q_{p_{\Theta'}}$
and $\Theta_3:=\eta_{\Theta'}^{-1}\Theta'\eta_{\Theta'}$ and deduce that
$$\eta_{\Theta}^{-1}Y^{\Theta}\cap \eta_{\Theta'}^{-1}Y^{\Theta'}\neq\emptyset.$$
This implies that the classes of $Y^{\Theta}$ and $Y^{\Theta'}$ in $\hH$ belong
to the same $R$-orbit. We have thus proved that
$l=|\hH/R|\leq |\pi(G)|\leq r.$
\end{proof}

\begin{lemma}
\label{lemma:diferencia-divisible}
Let $f:=(C_{\Jord}\Lambda(K_3)K_5-1)!$. The number
$\chi(X)-\chi(Y^{\gG})$
is divisible by $$\frac{d}{\GCD(d,f)}.$$
\end{lemma}
\begin{proof}
Let $(\cC,\phi)$ be a $G$-good triangulation of $X$.
Suppose that $\sigma\in \cC$ satisfies
$\phi^{-1}(|\sigma|)\subseteq Y$ and $|R_{\sigma}|\geq C_{\Jord}\Lambda(K_3)K_5$.
Since the triangulation is good, for any $x\in\phi^{-1}(|\sigma|)$
we have $|R_x|\geq C_{\Jord}\Lambda(K_3)K_5$. Let $x\in \phi^{-1}(|\sigma|)$ be any
point.
By Lemma \ref{lemma:B} there exists an abelian $K_3$-stable
subgroup $Z_x^{\stable}$ of $Z_x$ satisfying
$[Z_x:Z_x^{\stable}]\leq C_{\Jord}\Lambda(K_3)$.
The bound $|R_x|\geq C_{\Jord}\Lambda(K_3)K_5$ implies that $|\Pi(Z_x^{\stable})|\geq K_5$,
and, since
$x\in Y$, we have $X^{Z_x^{\stable}}\cap Y\neq\emptyset$.
This implies that $Z_x^{\stable}\in\gG$, so $x\in Y^{\gG}$.
Consequently, if $\sigma\in \cC$ satisfies $\phi^{-1}(|\sigma|)\subseteq Y$
but $\phi^{-1}(|\sigma|)\not\subset Y^{\gG}$ then
$|R_{\sigma}|<C_{\Jord}\Lambda(K_3)K_5$. In other words, the cardinal of
$R\cdot\sigma\subseteq \cC$
has $d/s$ elements, where $s\in\NN$ satisfies $s<C_{\Jord}\Lambda(K_3)K_5$ and is a divisor of
$d$. Hence, $|R\cdot\sigma|$ is divisible by $d/\GCD(d,f)$.

Since $\chi(X)=\chi(Y)$,
$$\chi(Y)-\chi(Y^{\gG})=\sum_{j\geq 0} (-1)^j
\sharp\left\{\sigma\in \cC\mid \phi^{-1}(|\sigma|)\subseteq Y,\,
\phi^{-1}(|\sigma|)\not\subset Y^{\gG},\,\dim \sigma=j\right\},$$
and the action of $R$ on $\cC$ preserves the dimension,
the result follows from considering separately the contribution of each
$R$-orbit in $\cC$ to the sum on the right hand side.
\end{proof}

Combining (\ref{eq:euler-Y-gG}) and Lemmas
\ref{lemma:l-at-most-r} and \ref{lemma:diferencia-divisible}
we deduce that
$$\frac{d}{\GCD(d,f)}\quad\text{ divides }\quad\chi(X)\left(1-\sum_{j=1}^l\frac{d}{e_j}\right),
\qquad\text{and}\qquad l\leq r.$$
Let us assume, without loss of generality, that $e_1\geq e_2\geq\dots\geq e_l$. Then
Lemma \ref{lemma:equacio-diofantina} below implies that
$e_1\geq d C_{\delta}^{-1},$
where
$$C_{\delta}:=\max\{C_{\Delta}(l,\chi(X)s)\mid 1\leq l\leq r,\,1\leq s\leq f\}$$
for some universal function $C_{\Delta}\colon \NN\times\NN\to\NN$
(the number $f$ was defined in Lemma \ref{lemma:diferencia-divisible}).
In particular, $C_{\delta}$ depends on $X$ and $\delta$, but is independent of $G$ and
its action on $X$.

In other words, the stabilizer $R':=R_{h_1}$ of $h_1\in\hH$ has index at most $C_{\delta}$ in $R$.
Now the action of $R'$ on $Y$ fixes the intersection
$$Y_1:=\bigcap_{F\in\fF,\,[F]=h_1}F,$$
which by Lemma \ref{lemma:big-intersection} satisfies $\chi(Y_1)=\chi(X)$.
Let $R''=\{\rho\in R'\mid Y_1\subseteq Y^{\rho}\}$ and let $S:=R'/R''$.
The group $S$ acts naturally on $Y_1$.

\begin{lemma}
\label{lemma:bound-R''}
We have
$[R':R'']=|S|\leq (\chi(X)C_{\Jord}\Lambda(K_3)K_5)^r$.
\end{lemma}
\begin{proof}
We proceed by contradiction. Since
$|\pi(S)|\leq |\pi(R')|\leq |\pi(G)|\leq r$,
if the inequality of the lemma
does not hold then there exists some prime $p\in \pi(S)$
and a $p$-Sylow subgroup $S_p$ of $S$ with more than
$\chi(X)C_{\Jord}\Lambda(K_3)K_5$ elements.
Let $Z'=\Pi^{-1}(R')$ and let $\Pi_S\colon Z'\to S$ be the composition of
$\Pi\colon Z'\to R'$ with the projection $R'\to S$.
Let $Z'_p$ be a $p$-Sylow subgroup of $Z'$ satisfying
$\Pi_S(Z'_p)=S_p$.
Applying Lemma \ref{lemma:B} to the action
of $Z'_p$ on $X$ and the submanifold $Y_1\subseteq X$, we deduce that
there is an abelian $K_3$-stable subgroup $Q$ of $Z'_p$
such that $[Z'_p:Q]\leq \chi(X)C_{\Jord}\Lambda(K_3)$
and $Y_1^Q\neq\emptyset$.
Since $|\Pi_S(Z_p')|=|S_p|>\chi(X)C_{\Jord}\Lambda(K_3)K_5$,
it follows that $|\Pi_S(Q)|\geq K_5$, which implies
$|\Pi(Q)|\geq K_5$. Since $\emptyset\neq Y_1^Q\subseteq Y^Q$,
it follows that $Q\in\gG$. Since $Y_1^Q\neq\emptyset$,
the class of $Q$ in $\hH$ is equal to $h_1$. By the definition
of $Y_1$, we should have
$Y_1\subseteq Y^{Q}$, which implies $|\Pi_S(Q)|=\{1\}$, contradicting
$|\Pi_S(Q)|\geq K_5$. The proof is complete.
\end{proof}

Summing up, $R$ contains a subgroup $R''$ of index at most
$C_{\delta}(\chi(X)C_{\Jord}\Lambda(K_3)K_5)^r$ all of whose elements
fix every point in $Y_1$. Let $Z'':=\Pi^{-1}(R'')$, and let $y\in Y_1$
be any point. Then $[Z:Z'']\leq C_{\delta}(\chi(X)C_{\Jord}\Lambda(K_3)K_5)^r$
and every element of $Z''$ fixes $y$. By  (2) in Lemma \ref{lemma:linearization} and
Theorem \ref{thm:Jordan-classic} there is an abelian subgroup
$A_{\delta}\subseteq Z''$ of index at most $C_{\Jord}$
which is isomorphic to a subgroup of $\GL(\dim X,\RR)$.
We have
$$[Z:A_{\delta}]\leq C_T(\delta):=C_{\Jord}C_{\delta}(\chi(X)C_{\Jord}\Lambda(K_3)K_5)^r,$$
so the proof of the induction step is complete. Hence the statement T($\delta$)
has been proved for all values of $\delta$.

\subsection{Conclusion of the proof of Theorem \ref{thm:bounded-primes-CT}}
\label{ss:end-of-proof}
As mentioned earlier,
T($rC_{\depth}(X)$) implies that any finite group
$G$ acting in an ECT way on $X$ and satisfying $|\pi(G)|\leq r$
contains an abelian subgroup $B:=A_{rC_{\depth}(X)}$ of index at most $C_T(rC_{\depth}(X))$
which is isomorphic to a subgroup of $\GL(\dim X,\RR)$. By Theorem \ref{thm:abelian-generat},
there is a subgroup $A\subseteq B$ of index at most $C_{\dim X}$ which can be generated by
at most $[\dim X/2]$ elements. So the proof of Theorem \ref{thm:bounded-primes-CT} is now
complete.

\subsection{An arithmetic lemma}
The following lemma is exactly the same as Lemma 4.1 in \cite{Bu}. We include a proof
for completeness.

\begin{lemma}
\label{lemma:equacio-diofantina}
There exists a function $C_{\Delta}\colon \NN\times\NN\to\NN$ with the following
property. Suppose that $d,e_1,\dots,e_l,a\in\NN$ and $t\in\ZZ$ satisfy
$e_1\geq \dots\geq e_l$, each $e_j$ divides $d$, and
\begin{equation}
\label{eq:equacio-diofantina}
\frac{d}{e_1}+\dots+\frac{d}{e_l}-1=\frac{dt}{a}.
\end{equation}
Then $e_1\geq d/C_{\Delta}(l,a).$
\end{lemma}
\begin{proof}
Consider for any $(l,a)\in\NN^2$ the set
$\sS(l,a)\subset \NN^{l+1}\times\ZZ$ consisting of
tuples $(d,e_1,\dots,e_l,t)$ satisfying
(\ref{eq:equacio-diofantina}) and also
$e_1\geq\dots\geq e_l$ and $e_j|d$ for each $j$.
Define $C_{\Delta}\colon \NN\times\NN\to\NN$ recursively
as follows: $C_{\Delta}(1,a):=a$ and, for each $l>1$,
$C_{\Delta}(l,a):=\max\{C_{\Delta}(l-1,aj)\mid 1\leq j\leq al\}$
(in fact $C_{\Delta}(l,a)=C_{\Delta}(l-1,a^2l)$).
We prove that for any $(d,e_1,\dots,e_l,t)\in\sS(l,a)$ we have
$e_1\geq d/C_{\Delta}(l,a)$ using induction on $l$. For the
case $l=1$, suppose that $(d,e_1,t)\in\sS(1,a)$ and let $d=e_1g$,
where $g\in\NN$. Rearranging (\ref{eq:equacio-diofantina}) we deduce
that $g$ divides $a$, which implies $g\leq a$,
so $e_1=d/g\geq d/a=d/C_{\Delta}(1,a)$.
Now assume that $l>1$ and that the inequality has been proved for smaller
values of $l$. Let $(d,e_1,\dots,e_l,t)\in\sS(l,a)$.
Since each $e_j$ divides $d$, we have $d/e_j\geq 1$ for each $j$, so the left
hand side in (\ref{eq:equacio-diofantina}) is positive. This implies that $t\geq 1$.
Using $e_1\geq\dots\geq e_l$ we can estimate $d/a\leq l d/e_l$,
so $1\leq e_l\leq al$. Furthermore, (\ref{eq:equacio-diofantina}) implies
$$\frac{d}{e_1}+\dots+\frac{d}{e_{l-1}}-1=\frac{dt}{a}-\frac{d}{e_l}=
\frac{d(te_l-a)}{ae_l},$$
so $(d,e_1,\dots,e_{l-1},te_l-a)$ belongs to $\sS(l-1,aj)$ for some $1\leq j\leq al$.
Using the induction hypothesis we deduce that $e_1\geq d/C_{\Delta}(l-1,aj)\geq d/C_{\Delta}(l,a)$.
\end{proof}

\section{Proof of Theorem \ref{thm:main}}
\label{s:proof-thm:main}

Let $X$ be a manifold without odd cohomology, and let $X_1,\dots,X_r$
be its connected components. Let $G$ be a finite group acting effectively on $X$.
By Lemma \ref{lemma:Minkowski} we may assume, replacing if necessary
$G$ by a subgroup of uniformly bounded index, that $G$ acts on $X$ in a
ECT way. This implies that the action of $G$ preserves the connected components
of $X$, so for any $i$ we have a map $G\to\Diff(X_i)$. Let $G_i\subset\Diff(X_i)$
be the image of this map, and let $\pi_i:G\to G_i$ be the natural projection.
The fact that $G$ acts effectively on $X$ means that $\bigcap_i\Ker\pi_i=\{1\}$.

Combining Theorem \ref{thm:bounded-primes} and
Theorem \ref{thm:MT} (taking
$\mathcal C$ to be the collection of finite subgroups of $\Diff(X_i)$),
we deduce that for every $i$ there is an abelian subgroup $A_i\subseteq G_i$
such that $[G_i:A_i]\leq C_i$, where $C_i$ depends only on $X_i$.

Let $A':=\bigcap_i\pi_i^{-1}(A_i)$. Then $[G:A]\leq \prod_iC_i$. Since
each $A_i$ is abelian, for any $a,b\in A'$ we have $\pi_i([a,b])=[\pi_i(a),\pi_i(b)]=1$
for every $i$. Since $\bigcap_i\Ker\pi_i=\{1\}$, we deduce that $[a,b]=1$, so $A'$ is abelian.
To finish the proof of Theorem \ref{thm:main}, apply
Theorem \ref{thm:main-abelian} to $A'$.

\end{document}